\documentclass[11pt]{amsart}

\usepackage{amssymb, graphicx}
\usepackage{graphicx}

\usepackage[all]{xy}

\usepackage{amsfonts}
\usepackage{amsmath}
\usepackage{latexsym}
\usepackage{amscd}
\usepackage{verbatim}

\usepackage{microtype}

\newcommand{\cA}{\mathcal{A}}

\newcommand{\T}{\mathcal{T}}

\newcommand{\Z}{\mathcal{Z}}

\newcommand{\mf}{\mathfrak}
\newcommand{\g}{\mf{g}}

\newcommand{\im}{{\rm im}}
\newcommand{\id}{{\rm id}}

\newcommand{\tnz}{\otimes}

\newcommand{\ksi}{\xi}
\newcommand{\tr}{\mathrm{tr}}

\allowdisplaybreaks[1]

\newtheorem{theorem}{Theorem}[section]
\newtheorem{proposition}[theorem]{Proposition}
\newtheorem{lemma}[theorem]{Lemma}
\newtheorem{corollary}[theorem]{Corollary}
 
\theoremstyle{remark}
\newtheorem{remark}[theorem]{Remark}
\newtheorem{definition}[theorem]{Definition}            
\newtheorem{example}[theorem]{Example}   
\newtheorem{question}[theorem]{Question}

\newcommand{\R}{\mathfrak{R}}
\newcommand{\A}{\mathcal{A}_R }
\newcommand{\B}{\mathcal{B}}
\newcommand{\N}{\mathbb{N}}
\newcommand{\C}{\mathcal C}

\newcommand{\QQ}{ \mathcal{C}\bigl\langle X \bigr\rangle }

\begin{document}
\title[Quasi-identities on matrices and the Cayley-Hamilton polynomial]{Quasi-identities  on matrices and the Cayley-Hamilton polynomial}

\author{Matej Bre\v sar}
\author{Claudio Procesi}
\author{\v Spela \v Spenko}
\address{M. Bre\v sar,  Faculty of Mathematics and Physics,  University of Ljubljana,
 and Faculty of Natural Sciences and Mathematics, University
of Maribor, Slovenia} \email{matej.bresar@fmf.uni-lj.si}
\address{C. Procesi, Dipartimento di Matematica, Sapienza
 Universit\` a di Roma,  Italy}  \email{procesi@mat.uniroma1.it}
\address{\v S. \v Spenko,  Institute of  Mathematics, Physics, and Mechanics,  Ljubljana, Slovenia} \email{spela.spenko@imfm.si}

\begin{abstract} We consider certain functional identities  on the matrix algebra $M_n$ that are  defined similarly as the trace identities, except that  the ``coefficients"  are arbitrary polynomials, not necessarily those expressible by the traces.  The main issue is 
 the question of whether such an identity is a consequence of the Cayley-Hamilton identity. We show that the answer is 
  affirmative 
  in several special cases, and, moreover, for every such an identity $P$ and every central polynomial $c$ with zero constant term  there exists $m\in\N$ such that the affirmative  answer holds for $c^mP$. In general, however, the answer is negative. We prove that there exist  antisymmetric identities that do not follow from the Cayley-Hamilton identity, and give a complete description of    a certain family  of such identities.
\end{abstract}

\thanks{2010 {\em Math. Subj. Class.} Primary 16R60. Secondary 16R10, 16R30. }
\thanks{The first author was supported by ARRS Grant P1-0288.}
\maketitle
\section{Introduction}     Given a finite dimensional algebra $A$ and an  integer $m$ (or $\infty$),  let  $\mathcal C$ be the commutative  ring of polynomial functions on $m$ copies of $A$ and $ \mathcal C\langle X \rangle $  the free algebra  in $m$ variables $X=\{x_1,\ldots,x_m\}$. We call this algebra the algebra of {\em quasi-polynomials} of $A$. Elements of $\mathcal C\langle X \rangle$ can clearly be {\em evaluated} in $A$  and then the {\em quasi-identities of $A$} are those quasi-polynomials that vanish at all evaluations (see Section \ref{sec1}).

In this paper we consider the fundamental case where  $A=M_n(F)$, the algebra of matrices. We will see that quasi-identities  appear in a natural way  as linear relations  among the noncommutative polynomial functions on $A$. In this sense the theory of quasi-identities is a worthwhile generalization to the theory of polynomial identities of $A$.
\vspace{0.13cm}

 Quasi-identities  appear as  a class of   {\em functional identities}. Let us give a brief  background on this more general notion.
 %, which will not be discussed further in the paper.
A functional identity   is an identical relation in a ring that, besides  arbitrary elements  that appear in a similar fashion as in a polynomial identity, also involves arbitrary functions which are considered as unknowns. 
The goal of the general functional identity theory is to describe these functions. 
It has 
  been developed with  applications in mind. 
Starting with the solution of a long-standing Herstein's problem on Lie isomorphisms in 1993 \cite{B}, functional identities have since turned out to be applicable to various problems in noncommutative algebra, nonassociative algebra, operator theory, functional analysis, and   mathematical physics. We refer the reader to the book \cite{FIbook} for an account of  functional identities and their applications.

Given a functional identity,  one usually first finds its ``obvious" solutions, i.e., those functions that satisfy this identity for formal reasons, independent of the structure of the ring in question. These are called the {\em standard solutions}. A typical result states that either the standard solutions are in fact the only possible solutions or the ring has some special properties, like satisfying a polynomial identity of a certain degree related to the number of variables. The existing theory of functional identities, as surveyed in \cite{FIbook}, thus gives definitive results for a large class of noncommutative rings, but, paradoxically, tells us nothing about the basic example of a noncommutative ring, i.e., the matrix algebra $M_n=M_n(F)$ (unless $n$ is big enough).   This is reflected in applications -- one usually has to exclude $M_n$ (for ``small"  $n$) in  a variety of results whose proofs depend on the general theory of functional identities, although 
by the nature of these results one can conjecture that this exclusion is unnecessary (see \cite{BB, BBS, BBCM, BC2, BF} for typical examples).   The problem with $M_n$ is that it allows {\em nonstandard solutions}. Their description seems to be a much harder problem than the description of standard solutions. Moreover, it is not clear what methods could be of use.
 
To the best of our knowledge, the recent paper \cite{BS} is the first work giving complete results on functional identities on $M_n$. However, it treats only 
functional identities in one variable. 
\vspace{0.1cm}

In this paper we restrict ourselves to the study of quasi-identities, which are important examples of functional identities in several variables.  
 Quasi-polynomials, also called Beidar polynomials in some papers, were introduced  in 2000 by Beidar and Chebotar  \cite{BC}, and have since played a fundamental role in  the theory of functional identities  and its applications.
 %, e.g. in the solution of Herstein's Lie map conjectures \cite{BBCM}. 
 Standard solutions of quasi-identities can be  very easily described: all coefficient functions must be $0$ (cf. \cite[Lemma 4.4]{FIbook}). 
The Cayley-Hamilton theorem gives rise to a basic example of a quasi-identity on the matrix algebra $M_n$ with nonstandard solutions. We call it the  Cayley-Hamilton  identity. The main theme of this paper is the following   question to which we have addressed ourselves: 
\vspace{0.13cm}

{\bf Question.} {\em Is every quasi-identity of $M_n$ a consequence of the Cayley-Hamilton  identity?}

 (A more accurate formulation will be given in the next section.)
\vspace{0.13cm}

 An important  motivation  for this question is the well-known  theorem, proved independently by Procesi \cite{Proc} and Razmyslov \cite{Raz},  saying that the answer to such a question is positive 
 for the related {\em trace identities}. Further,  from the main result of \cite{BS} it is evident that  nonstandard solutions of functional identities in one variable follow from the Cayley-Hamilton identity. 
 
 The main goal of this paper is to show that the answer to the above question is {\bf negative} in general.  The space $  \mathfrak{I}_n/(Q_n) $ of quasi-identities modulo the subspace of those quasi-identities which follow from the Cayley-Hamilton identity is determined in two steps by Proposition \ref{ls} and by Corollary \ref{chi}  through the exact sequence \eqref{kk}.  This already points out the geometric nature of the question, one sees that $ \mathfrak{I}_n/(Q_n) $ is an interesting invariant of the quotient map of the action of the projective linear group acting, by simultaneous conjugation, on the space of $m$-tuples of matrices. The space $  \mathfrak{I}_n/(Q_n) $  appears as a module on the quotient variety, supported on the singular set, cf. Theorem \ref{cp2}.   
Still  the complexity of this quotient map makes it difficult to describe this module and even to decide in a simple way if it is nonzero. 
Thus we restrict ourselves to describing a  particular subspace  of this module.  
We show that there exist antisymmetric quasi-identities of $M_n$
of degree $n^2$ that are not a consequence of the Cayley-Hamilton  identity. In fact in Theorem \ref{main}, which is the main result of the paper,   we give a precise description of those quasi-identities which transform under the linear group as the adjoint representation.
%Thus  the answer to the above question is {\bf negative} in general.  \smallskip
 In order to achieve this result, we first have   Theorem \ref{thb} of independent interest, which describes
 % a remarkable property, in the case $A_n=(\bigwedge M_n^*\otimes M_n)^G$, of 
 the way in which the adjoint representation of  the simple Lie algebra $\mathfrak g$ of traceless matrices sits in the exterior algebra $\bigwedge \mathfrak g$ of the same  Lie algebra $\mathfrak g$. 
 The discovery of this remarkable phenomenon has been the starting point for a general theorem for all simple Lie algebras, as shown in \cite{DPP}, and it also  gives  an insightful explanation of the basic theorem on identities of matrices, namely the Amitsur-Levitzki identity \cite{PA}.\smallskip

{\em Outline.} The paper is organised as follows.  In Section \ref{sec1} we collect notation, recall some basic facts of the theory of polynomial identities of matrices, and develop some basic formalism on functional identities of matrices.
 
   In Section \ref{sec2} we give the  structure of the main object of study, the  space $  \mathfrak{I}_n/(Q_n) $ of quasi-identities modulo the subspace of those quasi-identities which follow from the Cayley-Hamilton identity, in terms of the basic exact sequence \eqref{kk}. 
   From this we   prove   that for every quasi-identity $P$ and every central polynomial $c$ with zero constant term  there exists $m\in\N$ such that $c^m P$ is a  consequence of the Cayley-Hamilton  identity. This implies that the module $  \mathfrak{I}_n/(Q_n) $  is supported on the singular set. 
   
  Section \ref{sec3} is the most technical part  in  which a detailed study of  the  antisymmetric quasiidentities in degree $n^2$ is performed and the main Theorem \ref{main} is proved.

   Finally, Section \ref{sec4} is devoted to  positive results to the above question in various special cases. Thereby we  indicate that finding a quasi-identity that is not a consequence of  the Cayley-Hamilton  identity  can not be achieved in a simple minded way. 
  
 We finish the paper by giving a positive solution to Specht problem for quasi-identities, i.e., we  show that the 
   T-ideal of quasi-identities is finitely generated.

\section{Preliminaries }\label{sec1}
\subsection{Theory of identities\label{PI}}
 Before we enter in the main theme of this paper let us quickly review some basic facts  of the classical theory of identities. 
For more details we refer the reader to \cite{dr-for,for4,Proc3,Row}.

\subsubsection{Polynomial identities} Polynomial identities  appear in the formalism of {\em universal algebra}.  Whenever we have some category of algebras which admits {\em free algebras}  one has the concept of identities  in $m$ variables (where $m$ can also be $\infty$),   of an algebra $A$. That is the ideal of the free algebra $\mathcal F_m$ in $m$ variables $x_k$ formed by those element which vanish for all evaluations of the  variables $x_k$ into elements $a_k\in A$. An ideal of identities is a T-ideal, i.e., an ideal of the free algebra closed under substitution of the variables $x_k$ with elements  $H_i$ of the same free algebra. It is easily seen that a T-ideal $\mathcal I\subset \mathcal F_m$ is automatically the ideal of identities of an algebra, namely  $\mathcal F_m/ \mathcal I$.  Recall that we say that a T-ideal $\mathcal I$ is generated as T-ideal by a subset $I$ if it is the minimal T-ideal containing $I$. That is, it is generated as an ideal by all subsets obtained from $I$ applying substitution of variables with elements of the free algebra.

Of particular interest is  the case of noncommutative associative algebras over a field $F$,  for which we assume, for simplicity, that 
$${\rm char}(F)=0$$
(this assumption will be used throughout the paper without further mention).
 In this case the free algebra in  $m$ variables $x_k$ is the usual algebra of noncommutative polynomials with basis the {\em words} in the variables $x_k$. For $m =\infty$ we set
\begin{equation}
\label{lalf}{X} := \{x_k\,|\, k=1,2,\ldots\},\quad F\langle X\rangle\quad \text{the free algebra}. 
\end{equation}
In this case a particularly interesting example is the theory of polynomial identities of  the algebra $M_n=M_n(F)$ of all $n\times n$ matrices over  the field $F$.  An implicit description of these identities is given through the {\em algebra of generic matrices}.

We fix an integer $n\ge 1$, and set
\begin{equation}
\label{ilC}\mathcal{C}:=  F\bigl[x_{ij}^{(k)}\,|\, 1\le i,j\le n, k=1,2,\ldots\bigr].
\end{equation}
  This commutative polynomial ring is the algebra of polynomial functions  on sequences of matrices.
    Inside the matrix algebra $M_n(\mathcal C)$  we can define the {\em generic matrices}  $\xi_k$ where $\xi_k$ is the matrix with entries the  variables $x_{ij}^{(k)}$. It is then easily seen, since $F$ is assumed to be infinite,  that the ideal of polynomial identities of  $M_n(F)$  is the kernel of the evaluation map $x_k\mapsto \xi_k$.
  
 The $F$-subalgebra  of $M_n(\mathcal C)$ generated by the $\xi_k$, i.e., the image of the free algebra under this evaluation, is the free algebra in the $\xi_k$ in the category of noncommutative algebras satisfying the identities of   $M_n(F)$. This algebra has been extensively studied although a very precise description is available only for $n = 2$. We shall denote it by   $F\langle \xi\,|\,n\rangle$ or just   $F\langle \xi_k \rangle$ if the integer $n$ is fixed, 
% $F\langle \xi_k \rangle $  
and call  it the {\em algebra of generic matrices}.

\subsubsection{Trace identities}\label{TR} When dealing with matrices in characteristic 0, it is useful to think that they form an  algebra with a further  unary operation the {\em trace}, $x\mapsto \tr(x)$. One can formalize this as follows.

An {\em algebra with trace}  is an algebra $\R$ equipped with an additional structure, that is a linear map $\tr:\R\to \R$ satisfying the following properties
$$ \tr(ab)=\tr(ba),\quad a\,\tr(b)=\tr(b)\,a,\quad \tr(\tr(a)b)=\tr(a)\tr(b) $$
for all $a,b\in \R.$
The notion of a morphism between algebras with trace is then obvious and such algebras form a category which contains free algebras. 

 In this case the free algebra is the algebra of noncommutative polynomials with   basis the {\em words} in the variables $x_k$ but over the polynomial ring $\mathfrak T$  in the infinitely many variables  $\tr(M)$, where $M$ runs over all possible words considered equivalent under cyclic moves  (i.e., $ab\sim ba$).
  
  In this setting  again the {\em trace identities} of matrices are the kernel of the evaluation of the free algebra  into the generic matrices, but now the image is the subalgebra $ \mathcal T_n\langle \xi_k\rangle$ of  $M_n(\mathcal C)$  generated by the generic matrices and the algebra $\mathcal T_n$, the image of $\mathfrak T$,  generated by all traces of the  monomials in the $\xi_k$.\smallskip
  
  It is a remarkable fact that in this setting  both the trace identities and the {\em free} algebra $\mathcal T_n\langle \xi_k\rangle $  can be interpreted in the language of the first and second fundamental theorem for matrices (FFT and SFT).
  
  We have the projective linear group $G:=PGL(n,F)$ acting by conjugation on matrices and hence also on sequences of matrices, and we have (see \cite[Chapter 11]{Proc3}):
  
  \begin{theorem}\label{FST}
  \begin{enumerate}
  \item[
FFT:] The algebra $\mathcal T_n  $ is the algebra of $G$-invariant polynomial functions on the space of  sequences of matrices.
The algebra $\mathcal T_n\langle \xi_k\rangle $ is the algebra of $G$-equivariant polynomial maps from  sequences of matrices to matrices.

\item[
SFT:] The ideal of trace identities on $n\times n$ matrices is generated, as a T-ideal, by the Cayley-Hamilton polynomial.\end{enumerate}
\end{theorem}
 Another way of stating the FFT is by noticing that $G$ acts on $\mathcal C$  by $gf(x):=f(g^{-1}x)$ and on $M_n(F)$ by conjugation, hence it acts on $M_n(\mathcal C)= \mathcal C\otimes_FM_n(F)$ and we have
 $$\mathcal T_n\langle \xi_k\rangle = M_n(\mathcal C)^G,\quad \mathcal T_n=  \mathcal C^G.$$ Notice that as soon as $n\geq 2$  the algebra $\mathcal T_n$ is the center of $\mathcal R_n$.
 
 The FFT  for matrices is essentially classical, as for the SFT, Procesi \cite{Proc} and Razmyslov \cite{Raz} proved that  the T-ideal of trace identities of $M_n$ is generated by $Q_n(x_1,\ldots,x_n)$ and $\tr(Q_n(x_1,\ldots,x_n)x_{n+1})$, where $Q_n$ is the multilinear {\em Cayley-Hamilton polynomial}. Let us recall that 
  the  Cayley-Hamilton polynomial is 
$$
q_n=q_n(x_1) = x_1^n  +  \tau_1(x_1)x_1^{n-1} + \cdots + \tau(x_1). 
$$
 As it is well-known, each $\tau_i(x_1)$ can be expressed (in characteristic 0)  as a $\mathbb{Q}$-linear combination of the products of  ${\rm tr}(x_1^j)$.  
Evaluating in $M_n(C)$ we have $\tau_1(\xi_1) =- {\rm tr}(\xi_1) =- (x_{11}^{(1)} + \cdots + x_{nn}^{(1)}), \ldots, \tau_n(\xi_1)=(-1)^n\det(\xi_1)$. 
Now, $Q_n(x_1,\ldots,x_n)$ denotes the multilinear version of $q_n(x_1)$ obtained by full polarization. Recall that it can be written as
\begin{equation}
\label{ch}Q_n:=\sum_{\sigma\in S_{n+1}}\epsilon_\sigma\phi_\sigma(x_1,\ldots,x_n)
\end{equation} where $\epsilon_\sigma=\pm 1$ denotes the sign of the permutation $\sigma$, while  $\phi_\sigma$ is defined using the cycle decomposition of $$\sigma=(i_1,\ldots, i_{k_1})(j_1,\ldots ,j_{k_2})\ldots (u_1,\ldots ,u_{h})(s_1,\ldots s_{k }, n+1)$$ as
 $$\phi_\sigma(x_1,\ldots,x_n) =\tr(x_{i_1} \cdots x_{i_{k_1}})\tr(x_{j_1} \cdots x_{j_{k_2}})\cdots \tr(x_{u_1} \cdots x_{u_{h}})x_{s_1}\cdots x_{s_{k }}.$$
 Thus, for example, $$Q_2(x_1,x_2) = x_1x_2 + x_2x_1 -\tr(x_1)x_2 - \tr(x_2)x_1 + \tr(x_1)\tr(x_2)- \tr(x_1x_2) .$$
Note that $Q_n(x_1,\ldots,x_n)$ is symmetric, i.e.,  $Q_n(x_1,\ldots,x_n) = Q_n(x_{\sigma(1)},\ldots,x_{\sigma(n)})$ for every permutation $\sigma$, and that $q_n(x_1) = \frac{1}{n!} Q_n(x_1,\ldots,x_1)$. By a slight abuse of terminology, we will call both $Q_n$ and $q_n$ the  Cayley-Hamilton polynomial, or, when associated with $M_n$, the {\em Cayley-Hamilton identity}. In view of the terminology introduced below, more accurate names in the setting of this paper may be the Cayley-Hamilton quasi-polynomial (resp. quasi-identity), but  we omit ``quasi" for simplicity. 

\subsubsection{Central polynomials}\label{cepo}
Recall that an element of the free algebra $F\langle X\rangle$ is a {\em central polynomial} for $n\times n$ matrices if  it takes scalar values under any evaluation into matrices.  It is then clear that the center, denoted  $\Z_n$, of   the algebra  $F\langle \xi_k \rangle $  of generic matrices, is the image of the set of central polynomials.  A basic discovery based on the existence of central polynomials found independently by Formanek \cite{formanek3} and Razmyslov \cite{Raz73} is that the center $\Z_n$ is rather large.  Then fundamental theorems of PI theory tell us that the central  quotient of   $F\langle \xi_k \rangle $ and of  
   $\mathcal T_n\langle \xi_k\rangle  $ coincide  and as soon as $n\geq 2$  give rise to a division algebra of rank $n^2$ over  its center which is the field of quotients  of both $\Z_n$ and $\mathcal T_n$.
   
  % In fact a much more precise Theorem is true.

The following theorem gathers together some known facts, but we recall them for completeness. 
%, that is the evaluation in generic matrices of a central polynomial with no constant coefficient.  We then have
\begin{theorem}\label{Proce}
 If $c\in \Z_n$ has zero constant term, then
 $F\langle \xi_k \rangle [c^{-1}]=\mathcal T_n\langle \xi_k\rangle [c^{-1}]$ is a rank $n^2$  Azumaya algebra over its center $\Z_n[c^{-1}]=\T_n[c^{-1}]$.
Moreover, $$M_n(\mathcal{C}[c^{-1}])\cong \mathcal{C}[c^{-1}]\otimes_{\Z_n[c^{-1}]} \mathcal T_n\langle \xi_k\rangle [c^{-1}].$$
\end{theorem} 

\begin{proof}
As it is well known and easy to see,
 $c$ is an identity of $M_{n-1}$.
% in fact embed  $M_{n-1}(F)$ into $M_n(F)$ (an embedding which does not preserve 1).  When we evaluate $c$ in $M_{n-1}(F)$ we must get a scalar matrix of $M_n(F)$  but only 0 has this property.
%\smallskip
Since $c$ is invertible in $F\langle \xi_k \rangle [c^{-1}]$ and $\mathcal T_n\langle \xi_k\rangle [c^{-1}]$, these two algebras cannot have  nonzero quotients  satisfying the identities of   $M_{n-1}$.
It follows from the Artin-Procesi theorem that $F\langle \xi_k \rangle [c^{-1}]$ and  $\mathcal T_n\langle \xi_k\rangle [c^{-1}]$ are Azumaya algebras over their centers of rank $n^2$. These centers are clearly  $\Z_n[c^{-1}]$ and $\T_n[c^{-1}]$. By general properties, the reduced trace of  $x\in F\langle \xi_k \rangle [c^{-1}]$ is just the trace of $x$ considered as a matrix in  $F\langle \xi_k \rangle [c^{-1}]$, and 
$ F\langle \xi_k \rangle [c^{-1}]$ is closed under the reduced trace. Hence every element in $\mathcal T_n\langle \xi_k\rangle [c^{-1}]$ is contained in $F\langle \xi_k \rangle [c^{-1}]$.  Accordingly,   $F\langle \xi_k \rangle [c^{-1}]=\mathcal T_n\langle \xi_k\rangle [c^{-1}]$ and  $\Z_n[c^{-1}]=\T_n[c^{-1}]$.

%Now  it is simple to verify that the only elements in $M_n(\mathcal{C}[c^{-1}])$ which commute with $F\langle \xi_i \rangle $ are the scalar matrices and thus by a general property of Azumaya algebras  $M_n(\mathcal{C}[c^{-1}])=\mathcal{C}[c^{-1}]\otimes_{\Z_n[c^{-1}]} F\langle \xi_i \rangle [c^{-1}]$.  
Recall a standard fact (see \cite{A-G1} and  \cite{A-G2} or \cite[Theorem 2.8]{Salt}) that if $R\subseteq S$, $R$ is an Azumaya
%\marginpar{no hypothesis is needed on $S$, it can be any algebra, in fact if I recall this is really a theorem on bimodules}
 algebra and the center $Z(R)$ of $R$ is contained in the center $Z(S)$ of $S$, then
 $S\cong R\otimes_{Z(R)} R'$ where $R'$ is the centralizer of $R$ in $S$. Taking $\mathcal T_n\langle \xi_k\rangle [c^{-1}]$ for $R$ and  $M_n(\C[c^{-1}])$ for $S$ we obtain the last assertion of the theorem. 
\end{proof}
We should remark that this theorem has a geometric content.  Let $F$ be algebraically closed. If we fix the number of generic matrices to a finite number $m$,  we have the action by simultaneous conjugation of $G:=PGL_n(F)$ on the affine space $M_n(F)^m$. By geometric invariant theory  the algebra $\T_n$  is the coordinate ring of the categorical quotient $M_n(F)^m//PGL_n(F)$, a variety parameterizing the closed orbits, which correspond to isomorphism classes of semisimple representations of dimension $n$ of the free algebra in $m$ generators (cf. \cite{A}).

The action of the projective group $G$ is free on the open set of irreducible representations and the complement of this open set is exactly the subvariety of $m$-tuples of matrices where all central polynomials with no constant term vanish.  The Azumaya algebra property reflects this geometry.  Except for the special case $m=n=2$  the variety $M_n(F)^m//PGL_n(F)$ is smooth exactly on this open set and the quotient map  $M_n(F)^m\to M_n(F)^m//PGL_n(F)$ is not flat over the singular set (cf. \cite{L1}). As we shall see these singularities are in some sense measured   by the quasi-identities  of matrices modulo those which are a consequence of the Cayley Hamilton identity, see the exact sequence \eqref{kk}.  This will be described as a module $ \mathfrak{I}_n/(Q_n) $ supported in the singular part of the quotient variety. On the other hand to prove that this module is indeed nontrivial is quite difficult and although we will show this, we only have a partial description of this phenomenon, the description of the {\em antisymmetric part of the module}.  

\subsection{Quasi-identities}  
The purpose of this section is to introduce the setting and record some easy results  on the main theme of this paper, {\em quasi-identities}. Let us point out, first of all, that we will consider our problems exclusively  on the algebra $M_n=M_n(F)$. 
Various problems on functional identities studied in \cite{FIbook} can be solved 
 for quite general classes of rings, but  the study of nonstandard solutions is of a different nature and  confining to  matrices seems  natural in this context.

We will define a quasi-polynomial in a slightly different way than in \cite{BC} and \cite{FIbook}.  Our definition   is not restricted to the multilinear situation, and, on the other hand,  
is adjusted for applications  to the matrix algebra $M_n$. 

We use the notations of \eqref{lalf}  and \eqref{ilC}.
 A {\em quasi-polynomial} is an element of the algebra
   $ \mathcal{C}\bigl\langle X \bigr\rangle :=\mathcal{C}\bigl\langle X \bigr\rangle,$  the free algebra in the variables $X$ with coefficients in the polynomial algebra of functions on $M_n(F)^{|X|}$.

Thus, a quasi-polynomial is a polynomial in the noncommuting indeterminates $x_k$ whose coefficients are ordinary polynomials in the commuting  indeterminates $x_{ij}^{(k)}$, coordinates of the space $M_n(F)^{|X|}$. 
%The reasons for  labeling  indeterminates in such a way will  
A quasi-polynomial $P$ can be therefore uniquely written as 
$$
P=\sum \lambda_M M,
$$
where $M$ is a noncommutative monomial in the $x_k$'s and $\lambda_M$ is a commutative polynomial in  the $x_{ij}^{(k)}$'s, that is a polynomial function on sequences of matrices. Of course, $P$ depends on finitely many $x_k$'s and finitely many $x_{ij}^{(k)}$'s. We can therefore write 
$$P=P(x_{11}^{(1)}, \ldots,x_{nn}^{(1)},\ldots, x_{11}^{(m)},\ldots,x_{nn}^{(m)}, x_1,\ldots,x_m)$$
for some $m$. 
%We define the {\em evaluation} of $P$ at an $m$-tuple $A_1,\ldots,A_m\in M_n$, $P(A_1,\ldots,A_m)$,  by substituting $A_k$ for $x_k$ and  $a_{ij}^{(k)}$ for $x_{ij}^{(k)}$, where $A_k= (a_{ij}^{(k)})$. If $P(A_1,\ldots,A_m)=0$
%for all $A_1,\ldots,A_m\in M_n$, then we  say that $P$ is a {\em quasi-identity} of $M_n$. It is convenient to use a more suggestive notation and write $\lambda_M(x_{11}^{(1)}, \ldots,x_{nn}^{(1)},\ldots, x_{11}^{(m)},\ldots,x_{nn}^{(m)})$ as 
%$\lambda_M(x_1,\ldots,x_m)$, i.e., consider the polynomial $\lambda_M$ as a polynomial function $M_n^m\to F$. Hence 
%we can write $P$ as $P(x_1,\ldots,x_m)$. We define a {\em T-ideal} of $ \mathcal{C}\bigl\langle X \bigr\rangle $ as an ideal $\mathcal I$ such that if $P(x_1,\ldots,x_m)\in\mathcal I$, then  $P(H_1,\ldots,H_m)\in\mathcal I$ for all $H_1,\ldots,H_m\in  \mathcal{C}\bigl\langle X \bigr\rangle $. The set $\mathfrak{I}_n$ of all quasi-identities of $M_n$ clearly forms a T-ideal of $ \mathcal{C}\bigl\langle X \bigr\rangle $.
%===============================\marginpar{I do not understand}
It is possible to put this setting in the framework of universal algebra, but we shall limit to the following easy facts.

\subsubsection{Substitutional rules} Commutative indeterminates  $x_{ij}^{(k)}$  have a {\em substitutional rule}, that is  given as follows.  
We have a map $\Phi:x_k\mapsto \xi_k$ of  $ \mathcal{C}\bigl\langle X \bigr\rangle $ to $M_n(\mathcal{C})$  which maps $x_k$  to the corresponding generic matrix and is the identity on $\mathcal{C}$, so for each choice of  
$H\in  \mathcal{C}\bigl\langle X \bigr\rangle $ it makes sense to speak of $\Phi(H)_{ij}$, the $(i,j)$ entry of $\Phi(H)$. The substitution in $ \mathcal{C}\bigl\langle X \bigr\rangle $ should be understood as that one substitutes  $x_k\mapsto H_k\in  \mathcal{C}\bigl\langle X \bigr\rangle  $  and simultaneously $x_{ij}^{(k)}\mapsto    \Phi(H_k)_{ij}$. 
We define
\begin{definition}
A {\em T-ideal} of $ \mathcal{C}\bigl\langle X \bigr\rangle $ as an ideal that is closed under all such substitutions.
\end{definition} 
Also, it is convenient to use a more suggestive notation and write $\lambda_M(x_1,\ldots,x_m)$ for $\lambda_M(x_{11}^{(1)}, \ldots,x_{nn}^{(1)},\ldots, x_{11}^{(m)},\ldots,x_{nn}^{(m)})$, and hence
 $P(x_1,\ldots,x_m)$ for $P$.

We now define the {\em evaluation} of $P$ at an $m$-tuple $A_1,\ldots,A_m\in M_n$, $P(A_1,\ldots,A_m)$,  by substituting $A_k$ for $x_k$ and  $a_{ij}^{(k)}$ for $x_{ij}^{(k)}$, where $A_k= (a_{ij}^{(k)})$.\begin{definition}
If $P(A_1,\ldots,A_m)=0$
for all $A_1,\ldots,A_m\in M_n$, then we  say that $P$ is a {\em quasi-identity} of $M_n$. We denote by $\mathfrak{I}_n$ the set of all quasi-identities of $M_n$.
\end{definition} 
The set $\mathfrak{I}_n$ of all quasi-identities of $M_n$ clearly forms a T-ideal of $ \mathcal{C}\bigl\langle X \bigr\rangle $. As for polynomial or trace identities,  $\mathfrak{I}_n$  is the kernel of the $\mathcal C-$linear evaluation map from $ \mathcal{C}\bigl\langle X \bigr\rangle $ to $M_n(\mathcal C)$ mapping $x_k$ to the generic matrix $\xi_k$.

 Let $I$ denote the identity of $M_n(\mathcal C)$.  For convenience we  repeat the proof in 
  \begin{lemma}\label{gen1}
 The algebra $ \mathcal{C}\bigl\langle X \bigr\rangle /\mathfrak{I}_n$ is isomorphic to the subalgebra $\mathcal{C}\langle \xi_k\rangle$ of $M_n(\mathcal{C})$ generated by all generic matrices $\xi_k=(x_{ij}^{(k)})$, $k=1,2,\ldots$, and all $\lambda I$, $\lambda\in \mathcal{C}$.  
 \end{lemma}
  \begin{proof}
 Let $\Phi:  \mathcal{C}\bigl\langle X \bigr\rangle \to M_n(\mathcal{C})$ be the homomorphism determined by $\Phi(x_k) = (x_{ij}^{(k)})$ and $\Phi(\lambda)=\lambda I$ for $\lambda\in \mathcal{C}$.
It is immediate that  $ \ker \Phi\subseteq \mathfrak{I}_n$. Given $P=P(x_1,\ldots,x_m)\in \mathfrak{I}_n$ we have $P(A_1,\ldots,A_m)=0$ for all $A_i\in M_n$. Since char$(F)=0$, and hence $F$ is infinite, a standard argument shows that  $\Phi(P)=0$. Thus,  $ \ker \Phi= \mathfrak{I}_n$, and the result follows.
 \end{proof}

In fact the evaluation $\rho$ of the free algebra with trace to generic matrices with traces factors through  $ \mathcal{C}\bigl\langle X \bigr\rangle $
$$\rho:\mathfrak T\langle X \rangle\stackrel{\pi}  \longrightarrow  \mathcal{C}\bigl\langle X \bigr\rangle \to M_n(\mathcal C)$$ by evaluating the trace monomials $\tr(x_{i_1}\cdots x_{i_m})$  into $M_n(\mathcal C)$
 but keeping fixed the free variables.   The image of $\pi$ is the algebra   $\T_n\bigl\langle X \bigr\rangle$ of invariants  of the algebra  $\mathcal{C}\bigl\langle X \bigr\rangle$ with respect to the action of the projective group on the coefficients $\mathcal C$ and fixing the variables $X$.\smallskip
 
 Thus the image through $\pi$ of a trace polynomial can also be viewed as a quasi-polynomial $\sum \lambda_M M$, but such that every $\lambda_M$ is an invariant and thus can be expressed as a linear combination of the products of $\tr(x_{i_1}\cdots x_{i_m})$.

Every trace identity gives rise  to a quasi-identity of $M_n$, but a nontrivial element of  $\mathfrak T\langle X \rangle$  may very well map to 0 under $\pi$, so  a nontrivial trace identity may  correspond to a trivial quasi-identity. In view of the SFT for matrices we may again consider    the quasi-polynomial arising from the Cayley-Hamilton theorem, therefore it is natural to look in this context for a possible analogue of the SFT  for quasi-identities (cf. Theorem \ref{FST}).

 \begin{definition}
We shall say that a quasi-identity $P$ of $M_n$ is a {\em consequence  of the Cayley-Hamilton identity} if $P$ lies in  
 the T-ideal of $ \mathcal{C}\bigl\langle X \bigr\rangle $ generated by $Q_n$.
\end{definition} The question pointed out in the introduction thus asks  the following:
\vspace{0.13cm}

{\bf Main question.} {\em Is the T-ideal $\mathfrak{I}_n$  generated by $Q_n$?}

(Here we may replace $Q_n$ by $q_n$, as $q_n$ and $Q_n$ generate the same T-ideal.) 
\vspace{0.13cm} 

As we have already remarked, the ideal of quasi-identities $\mathfrak{I}_n$ is the kernel of the evaluation map $\Phi$ of  $ \mathcal{C}\bigl\langle X \bigr\rangle $ into $M_n(\mathcal C)$ mapping the variables to the generic matrices. 
In view of Lemma \ref{gen1} we  have a sequence of inclusion maps 
$$F\langle \xi_k \rangle \subset \mathcal T_n\langle \xi_k\rangle \subset  \mathcal C\langle\xi_k\rangle.$$    Our first remark is that, unlike  $F\langle \xi_k \rangle \cong F\langle  X \rangle/{\rm id}(M_n)$ and $\mathcal T_n\langle \xi_k\rangle $, $\mathcal C\langle\xi_k\rangle\cong \mathcal{C}\bigl\langle X \bigr\rangle /\mathfrak{I}_n $ is not a domain. This can be deduced  from Lemma \ref{gen2} below, but let us, nevertheless, give a simple concrete example.
 
 \begin{example} 
Note that none of 
 $$P_1 = x_{12}^{(2)}x_1   -x_{12}^{(1)}x_2 + x_{12}^{(1)}x_{22}^{(2)} - x_{22}^{(1)}x_{12}^{(2)}$$
 and
 $$P_2 = x_{12}^{(2)}x_1   -x_{12}^{(1)}x_2 + x_{12}^{(1)}x_{11}^{(2)} - x_{11}^{(1)}x_{12}^{(2)}$$
lies in $\mathfrak{I}_2$, but $P_1P_2$ does.
 \end{example}
 
 The center of $ \mathcal{C}\bigl\langle X \bigr\rangle /\mathfrak{I}_n$ is isomorphic to $\mathcal{C}$, which is a domain. We may therefore form the algebra of central quotients of $ \mathcal{C}\bigl\langle X \bigr\rangle /\mathfrak{I}_n$, which consists of elements of the form $\alpha R$ where $R\in  \mathcal{C}\bigl\langle X \bigr\rangle /\mathfrak{I}_n$ and $\alpha$ lies in
  $$\mathcal{K}:=  F\bigl(x_{ij}^{(k)}\,|\, 1\le i,j\le n, k=1,2,\ldots\bigr),$$
 the field of rational functions in $x_{ij}^{(k)}$  (cf. \cite[p. 54]{Row}). In order to describe this $\mathcal K$-algebra, we invoke the Capelli polynomials
 $$
 C_{2k-1}(x_1,\ldots,x_k,y_1,\ldots,y_{k-1}) := \sum_{\sigma\in S_k} \epsilon_\sigma x_{\sigma(1)}y_1 x_{\sigma(2)} y_2\cdots  x_{\sigma(k-1)}y_{k-1} x_{\sigma(k)},
 $$
  where $\epsilon_\sigma$ is the sign of the permutation $\sigma$.
 As it is well-known, $C_{2n^2-1}$ is a polynomial identity of every proper subalgebra of $M_n(E)$ but not of $M_n(E)$ itself, for every field $E$ \cite[Theorem 1.4.8]{Row}. 
 
 \begin{lemma}\label{gen2}
 The algebra of central quotients of $ \mathcal{C}\bigl\langle X \bigr\rangle /\mathfrak{I}_n$ is isomorphic to $M_n(\mathcal{K})$.
 \end{lemma}

\begin{proof} Since $C_{2n^2-1}$ is not a polynomial identity of $M_n(F)$, it is also not a  polynomial identity of  the $\mathcal K$-subalgebra of $M_n(\mathcal{K})$ generated by all generic matrices $(x_{ij}^{(k)})$, $k=1,2,\ldots$. But then this subalgebra is the whole algebra $M_n(\mathcal{K})$.
Now we can apply Lemma \ref{gen1}.
\end{proof}

%Lemma \ref{gen2} can be considerably refined, see 
%Theorem \ref{Proce} below. But even this simple version is of some use. 
We conclude this section with a small application. Define the image of $P=P(x_1,\ldots,x_m)\in \mathcal{C}\bigl\langle X \bigr\rangle $ as 
$$\im(P) = \{P(A_1,\ldots,A_m)\,|\,A_1,\ldots,A_m\in M_n\}.$$
It is an open question which subsets of $M_n$ can be images of noncommutative polynomials; cf. \cite{BMR, Sp}. Since, on the other hand, such an image is closed under conjugation  it follows that among linear subspaces of $M_n$ there are only four possibilities: $\{0\}$,  the space of all scalar matrices, the space of all trace zero matrices, and $M_n$. The situation with quasi-polynomials is strikingly different.

\begin{proposition} For every linear subspace $V$ of $M_n$
there exists $P\in \QQ$ such that $\im(P)=V$.
\end{proposition}

\begin{proof}
By taking the sums of quasi-polynomials in distinct indeterminates we see that it is enough to prove the theorem for the case where $V$ is one-dimensional, $V = FA$ for some $A\in M_n$. 
According to Lemma \ref{gen2}, we may identify $x_{11}^{(1)}A\in M_n(\mathcal K)$ with $\lambda^{-1}P_0$ where $0\ne\lambda\in \mathcal C$ and $P_0\in\QQ$.
Hence $\im(P_0)\subseteq FA$. Picking an indeterminate $x_{ij}^{(k)}$ of which $P_0$ is independent we thus see that $P= x_{ij}^{(k)}P_0$
satisfies $\im(P)=FA$.
\end{proof}
\section{Quasi-identities and the Cayley-Hamilton identity  \label{sec2}}

\subsection{Trace algebras and the Cayley-Hamilton identity}

We begin by reformulating our problem in the commutative algebra framework by using the result from \cite{Proc2}. Let us, therefore, recall the 
content of that paper.  We have already seen in \S \ref{TR} the notion of an algebra with trace, in particular for any commutative algebra $ A$ we consider $M_n( A)$ with the usual trace.\smallskip

For an algebra with trace $R$  and a number $n\in \mathbb N$, we define the  universal map into $n\times n$ matrices  as a  pair of a commutative algebra $\A$  and a morphism (of algebras with trace) $j: R\to M_n(\A)$  with the following universal property: for any other map (of algebras with trace) $f: R\to M_n(\B)$  with $\B$ commutative there is a unique map $\bar f:\A\to \B$ of commutative algebras making the diagram commutative
$$\xymatrix{ &R\ar@{->}[r]^{j}\ar@{->}[rd]^{f} &M_n(\A)\ar@{->}[d]^{M_n(\bar f)}\\&  & M_n(\B)  } $$ 
 The existence of such a universal map is  easily established, although in general it may be 0.

The main idea comes from   category theory, that is, from  representable functors.   We take the functor  from commutative algebras to sets  which associates to a commutative algebra $\B$
 the set of (trace preserving) morphisms  $\hom(R, M_n(\B))$ and want to prove that it is representable, i.e., that there is a commutative algebra $\A$ and a natural isomorphism $\hom(R, M_n(\B))\cong  \hom(\A,  \B) $.  
Then the identity map $1_{\A}\in \hom(\A, \A)$ corresponds to the universal map  $j\in \hom(R,M_n(\A))$. 

We have seen the category of algebras with trace has   free algebras which are the usual free algebras in  indeterminates $x_k$ to which we add a commutative algebra $\mathfrak T$ of {\em formal traces}. Then we see that the commutative algebra associated to a free algebra is the  polynomial algebra in indeterminates $x_{ij}^{(k)}$. The universal map maps $x_k$ to the generic matrix with entries $x_{ij}^{(k)}$  and   the formal traces   $\tr(x_{i_1}x_{i_2}\ldots x_{i_k})$ map to the traces of the corresponding monomials in the generic matrices. From a presentation of $R$ as a quotient of a free algebra one obtains a presentation of $\A$ as a quotient of the ring of polynomials in the $x_{ij}^{(k)}$.

If we consider now algebras over a field  $F$ (which is of characteristic 0 by the above convention)   we have that the group $G=GL_n(F)$  of invertible $n\times n$ matrices (in fact the projective group  $PGL_n(F)=GL_n(F)/F^*\,)$ acts on the algebra $\A$ and it  also acts by conjugation on $M_n(F)$,  so it acts diagonally on $M_n(\A)$. The main theorem of \cite{Proc2} says that 
\begin{theorem}\label{CHA}
The image of $j$ is the invariant algebra $M_n(\A)^G$ and the kernel of $j$ is the trace-ideal generated  by  the evaluations of the formal Cayley-Hamilton expression for the given $n$. In particular, if  $R$  satisfies
the $n$-th 
Cayley-Hamilton
identity, then $j$ is injective.
\end{theorem}

\subsubsection{The trace on $\mathcal C\langle X\rangle$}  Now we will apply this theory to $\mathcal C\langle X\rangle$. For this we need to make it into an algebra with trace. For reasons that will soon become clear,  let us write $\mathcal C_x$ for $\mathcal C$ and hence $\mathcal C_x\langle X\rangle$ until the end of this section.\begin{definition}
We define the trace $\tr:\mathcal C_x\langle X\rangle\to \mathcal C_x $ as  the $\mathcal C_x$-linear map satisfying $\tr(1)=n$ and mapping
 a monomial in the indeterminates $x_k$ into  the trace of the corresponding monomial in generic matrices $\xi_k$   in the indeterminates $x_{ij}^{(k)}$.
\end{definition}
In order to understand  what is the universal map of this algebra with trace into $n\times n$ matrices  we  introduce a second polynomial algebra  $\mathcal{C}_y =  F\bigl[y_{ij}^{(k)}\,|\, 1\le i,j\le n, k=1,2,\ldots\bigr]$.

The group $G$  acts on  $\mathcal C_x$   and  $\mathcal C_y$, and, by the FFT,    the invariants  are in both cases the invariants of matrices, that is the algebra generated by the traces of  monomials.  We identify the two algebras of invariants and call this algebra $\T_n$. 

Now  we  set $$\cA_n:=  \mathcal C_x\otimes_{\T_n} \mathcal C_y,$$  and let $\xi_k:=(y_{ij}^{(k)})$ denote the generic matrix in $M_n(\mathcal C_y)$.   Note that the algebra $\T_n\langle\ksi_k,k =1,2,\ldots \rangle$  may be identified to the algebra $\mathcal T_n\langle \xi_k\rangle $ of equivariant maps studied in Theorem  \ref{FST} and which has $\mathcal T_n$ as the center.

From now on let $j:\mathcal C\langle X\rangle\to M_n(\cA_n)$  denote the $\mathcal C_x$-linear map  which maps $x_k$ to the generic matrix $\xi_k=(y_{ij}^{(k)})$, and let $(Q_n)$ denote the T-ideal of $ \mathcal{C}\bigl\langle X \bigr\rangle $ generated by $Q_n$.  Since we are thinking of $\mathcal C_x$ as a coefficient ring, in the next proposition the action of $G$ on $\cA_n =  \mathcal C_x\otimes_{\T_n} \mathcal C_y,$ is by acting on the second factor $\mathcal C_y$. The action on $M_n(\cA_n)= M_n(F)\otimes_F\cA_n$ is  the tensor product action.

\begin{proposition}\label{ls} \begin{enumerate}
\item The map $j:\mathcal C_x\langle X\rangle\to M_n(\cA_n)=M_n(\mathcal C_x\otimes_{\T_n} \mathcal C_y)$ is the universal map into matrices.

\item The algebra $\mathcal C\langle X\rangle/(Q_n)$ is isomorphic to the algebra $M_n(\cA_n)^G=  \C_x\otimes_{\T_n}\mathcal T_n\langle \xi_k\rangle $.

\end{enumerate}\end{proposition}

\begin{proof}
By Theorem \ref{CHA},  (2) follows from (1) so it is enough to prove   that $j$ is the universal map. 

Take an algebra with trace $\B$. Let
us compute the representable functor $\hom(\mathcal C\langle X\rangle,M_n(\B))$. In order to give  a homomorphism  $\phi:\mathcal C\langle X\rangle\to M_n(\B) $ in the category of algebras with trace,  we have to choose arbitrary elements $a_{ij}^{(k)}\in \B$ for the images of the elements $x_{ij}^{(k)}$, and matrices $B_k=(b_{ij}^{(k)})$  for the images of the elements $x_k$. 

Moreover, if we consider the matrices $A_k:=(a_{ij}^{(k)})$ we need to  impose that the trace of each monomial formed by the $A_k$ equals the trace of the corresponding monomial formed by the $B_k$.

Now to give the $a_{ij}^{(k)}$ is the same as to give a homomorphism of  $\mathcal C_x$ to $\B$, and to give the $b_{ij}^{(k)}$ is the same as to give a homomorphism of  $\mathcal C_y$ to $\B$. The compatibility means that the restrictions of these two homomorphisms to the algebra  $\T_n$, which is  contained naturally in both copies, coincide. This is exactly  the description of a homomorphism of  $\mathcal C_x\otimes_{\T_n} \mathcal C_y$ to $\B$. Thus, $j$ is indeed the universal map.

Next observe that the action of $G$  is only on the factor  $ \mathcal C_y$. By Theorem \ref{CHA}  it follows that the kernel of $j$ is  equal to $(Q_n)$. Thus, 
 it remains to find  $M_n(\cA_n)^G$, the image of $j$.
Note that $M_n(\cA_n)= \mathcal C_x\otimes_{\T_n}M_n( \mathcal C_y)$ and that $G$ acts trivially on $\mathcal C_x$  while on $M_n( \mathcal C_y)$  it is the action used in the universal map of the free  algebra with trace (see \cite{Proc2} for details). By a standard argument on reductive groups we have
$M_n(\cA_n)^G=\mathcal C_x\otimes_{\T_n}M_n( \mathcal C_y)^G$, which is by the FFT equal to  $ \C_x\otimes_{\T_n} \mathcal T_n\langle \xi_k\rangle $.
%\marginpar{perhaps it is worth giving a name $\mathcal T_n\langle \xi_i\rangle := \T_n\langle\ksi_k,k =1,2,\ldots \rangle$ and stating before this theorem the theorem that $M_n( \mathcal C_y)^G= \mathcal T_n\langle \xi_i\rangle  $} 
 \end{proof}

 \begin{remark}
The algebra $\mathcal C_x\otimes_{\T_n} \mathcal C_y$, a fiber product,  contains a lot of the hidden combinatorics needed to understand functional identities.  It appears to be a rather complicated object as pointed out by some experimental computations carried out by H. Kraft (whom we thank), which show that even  for $n=2$  as soon as the number of variables is $\geq 3$ it is not an integral domain nor is it Cohen-Macaulay. This of course is due to the fact that the categorical quotient described by the inclusion $\T_n=\mathcal C^G$  has a rather singular behavior outside the open set parameterizing irreducible representations.
\end{remark}
We  have to introduce some more notation. As in the proof of the preceding proposition, let $\xi_k$ stand for the generic matrix $(y_{ij}^{(k)})$. Analogously, we write $\eta_k$ for the generic matrix $(x_{ij}^{(k)})$. There is a canonical homomorphism $$\pi:\C_x\otimes_{\T_n} \mathcal T_n\langle \xi_k\rangle  \to  \C_x\langle\eta_k,k=1,2,\ldots\rangle,$$  
$$\pi:\lambda\otimes f(\xi_1,\ldots,\xi_d)\mapsto \lambda f(\eta_1,\ldots,\eta_d)$$
 (note that by Lemma \ref{gen1} the latter algebra is nothing but  $ \mathcal{C}\bigl\langle X \bigr\rangle /\mathfrak{I}_n$).
%Due to Proposition \ref{ls} we will identify $\mathcal C\langle X\rangle/(Q_n)$ with $\C_x\otimes_{\T_n} \T_n\langle\ksi_k,k =1,2,\ldots \rangle$.

\begin{lemma}\label{TPr}
A quasi-identity $P$ of $M_n$ is not a consequence  of the Cayley-Hamilton identity if and only if $j(P)$ is a nonzero element of the kernel of $\pi$.
%The canonical homomorphism $\pi: \C_x\otimes_{\T_n} \T_n\langle\ksi_k,k =1,2,\ldots \rangle\to \C_x\langle\eta_k,k=1,2,\ldots\rangle$ is an isomorphism.
 %\end{enumerate}
%These conditions are fulfilled if the algebra $\mathcal \cA_n$ is a domain (actually, it is enough that elements of the form  $1\otimes q$ are not zero divisors in $\cA_n$).
\end{lemma}

\begin{proof}
Let 
 $\Phi:  \mathcal{C}\bigl\langle X \bigr\rangle \to \C_x\langle\eta_k,k=1,2,\ldots\rangle$ be the  homomorphism from Lemma \ref{gen1}, i.e., 
 $\Phi(x_k) = \eta_k$ and $\Phi(\lambda)=\lambda I$ for $\lambda\in \mathcal{C}_x$, and let $j:\mathcal C\langle X\rangle\to \C_x\otimes_{\T_n} \mathcal T_n\langle \xi_k\rangle $ be the universal map. Note that $\pi j = \Phi$. 
%Therefore $\pi$ is injective and hence an isomorphism if and only if $\ker j =\ker\Phi$. 
Since, by Proposition \ref{ls}, $\ker j$ is the T-ideal of $ \mathcal{C}\bigl\langle X \bigr\rangle $ generated by $Q_n$, and $\ker \Phi =\mathfrak I_n$, this implies the assertion of the lemma.
\end{proof}
 \begin{corollary}\label{chi}
The space   $  \mathfrak{I}_n/(Q_n) $, measuring quasi-identities modulo the ones deduced from $Q_n$, is isomorphic under the map induced by $j:\mathcal C\langle X\rangle/(Q_n)\to C_x\tnz_{\T_n}  \mathcal T_n\langle \xi_k\rangle $ to the kernel of the map $\pi$. That is, we have an exact sequence
\begin{equation}\label{kk}\begin{CD}
0@>>>  \mathfrak{I}_n/(Q_n)@>j>>\mathcal C_x\tnz_{\T_n}  \mathcal T_n\langle \xi_k\rangle @>\pi>>M_n(\mathcal C_x).
\end{CD}
\end{equation}
\end{corollary}

As an application of Theorem \ref{Proce} we have the following theorem on quasi-identities.

\begin{theorem}\label{cp2} Let $P$ be a quasi-identity of $M_n$. For every  central polynomial $c$ of $M_n$ with zero constant term there exists $m\in\N$ such that $c^m P$ is a consequence of the Cayley-Hamilton identity.
\end{theorem}

\begin{proof} %We  use the notation from the preceding section. 
Note that
$$
   (\C_x\otimes_{\T_n} \mathcal T_n\langle \xi_k\rangle )[c^{-1}] 
\cong  \C_x[c^{-1}]\otimes_{{\T_n}[c^{-1}]} \mathcal T_n\langle \xi_k\rangle [c^{-1}]
\cong M_n(\mathcal{C}_x[c^{-1}])
$$
by Theorem \ref{Proce} (the change of variables does not make any difference since $\C_x$ is canonically isomorphic to $\C_y$). This isomorphism is induced by $\pi$ introduced before Lemma \ref{TPr}. Therefore $(\ker \pi)[c^{-1}]=0$. Since every quasi-identity $P$ lies  in $\ker (\pi j)$ by Lemma  \ref{TPr},  there exists $m\in \N$ such that $c^mP=0$ in $\C_x\otimes_{\T_n}\T_n\langle\ksi_k,k =1,2,\ldots \rangle$, i.e., $c^mP$ is a consequence of the Cayley-Hamilton identity by  Proposition \ref{ls}.
%which does not follow from Cayley-Hamilton identity, i.e. if $P$ lies in the kernel of $\pi$, then t
\end{proof}

%Let us point out the  connection to the discussion from the preceding section. 

%\begin{corollary}
%If nontrivial central polynomials are not zero divisors in  $\C_x\otimes_{\T_n} \T_n\langle\ksi_k,k =1,2,\ldots \rangle$, then every quasi-identity of $M_n$ 
% lies in the T-ideal of $ \mathcal{C}\bigl\langle X \bigr\rangle $ generated by $Q_n$.
%\end{corollary}

%\begin{remark} Let $j$ be as in Section \ref{sec1}, and let $P$ and $c$ be as in  Theorem \ref{cp2}.  Then $j(c)^mj(P)=0$ in  $\C_x\otimes_{\T_n} \T_n\langle\ksi_k,k =1,2,\ldots \rangle$ by Proposition \ref{ls}. Therefore either $j(c)$ is a zero divisor in $\C_x\otimes_{\T_n} \T_n\langle\ksi_k,k =1,2,\ldots \rangle$ or $P$ is a consequence of the Cayley-Hamilton identity. 
%\end{remark}
%We do not know whether or not  Theorem \ref{cp} gives an optimal conclusion. It does not seem impossible that for some quasi-identities the involvement of a nontrivial central polynomial $c$ is necessary, but we were unable to find an example.

We have seen that $\ker\pi$ measures the space of quasi-identities modulo the ones deduced from $Q_n$. This is in fact a $\mathcal T_n-$module and, as we shall see, it is nonzero. What the previous theorem tells us is that this module is {\em supported}  in the closed set of  non-irreducible representations.

\section{Antisymmetric quasi-identities \label{sec3}}  

\subsection{Antisymmetric identities derived  from the Cayley-Hamilton identity}% In this section we study multilinear antisymmetric identities. Let us first introduce the appropriate setting for this.
By the {\em antisymmetrizer} we mean the operator that sends a multilinear expression $f(x_1,\ldots,x_h)$ into the antisymmetric expression $\frac{1}{h!}\sum_{\sigma\in S_h} \epsilon_\sigma f(x_{\sigma(1)},\ldots,x_{\sigma(h)})$, where $\epsilon_\sigma$ is the sign of  $\sigma$. For example, applying the antisymmetrizer to the noncommutative monomial $x_1\cdots x_h$ we get
the standard polynomial of degree $h$,  $S_h(x_1,\ldots,x_h)=\sum_{\sigma\in S_{h}}\epsilon_\sigma   x_{\sigma(1)} \dots x_{\sigma(h)},$
and up to scalar this is the only multilinear antisymmetric noncommutative polynomial of degree $h$.
 Further,
 applying the antisymmetrizer to the quasi-monomial 
$x_{i_1,j_1}^{(1)}\cdots x_{i_k,j_k}^{(k)} x_{k+1}\cdots x_{n^2}$
we get an antisymmetric quasi-polynomial, which is nonzero as long as the pairs $(i_l,j_l)$ are pairwise different, and is, because of the antisymmetry, an identity of every proper subspace of $M_n$, in particular of the space of trace zero $n\times n$ matrices. Replacing each variable $x_k$ by $x_k - \frac{1}{n} \tr(x_k)$, we thus get a quasi-identity of $M_n$. Our ultimate goal is to show that {\em not every such quasi-identity is a consequence of the Cayley-Hamilton identity}. For this we need several auxiliary results. We start by introducing the appropriate setting.

Let $A$ be a finite dimensional $F$-algebra with basis $e_i$,  and let $V$  be a vector space  over $F$. The set of  multilinear  antisymmetric functions from $V^k$ to $A$  is given by functions $F(v_1,\ldots,v_k)=\sum_iF_i(v_1,\ldots,v_k)e_i$  with $F_i(v_1,\ldots,v_k)$ multilinear  antisymmetric functions from $V^k$ to $F$,  in other words  $F_i(v_1,\ldots,v_k)\in \bigwedge^kV^*$.
This space can be therefore identified with $\bigwedge^kV^*\otimes A$.  Using the algebra structure of $A$ we have a wedge product of these functions: for $F\in \bigwedge^hV^*\otimes A,\ H\in \bigwedge^kV^*\otimes A$ we define
$$F\wedge H( v_1,\ldots,v_{h+k}):=\frac{1}{h!k!}\sum_{\sigma\in S_{h+k}}\epsilon_\sigma F (v_{\sigma(1)},\ldots,v_{\sigma(h)})H(v_{\sigma(h+1) },\ldots,v_{\sigma(h+k)})$$
It is easily verified that with this product 
the algebra of multilinear  antisymmetric functions from $V $ to $A$ is isomorphic to the tensor product algebra $\bigwedge V^*\otimes A$.
We will apply this to $V=A=M_n$. Again the group $G=PGL(n,F)$ acts  on these functions and it will be of interest to study the {\em invariant algebra} $$A_n:=(\bigwedge M_n^*\otimes M_n)^G.$$ %We denote this algebra by $A_n$.

If $N_n$ denotes the Lie algebra of trace zero $n\times n$ matrices, the multilinear and antisymmetric trace expressions  for such matrices can be identified with the invariants  $(\bigwedge N_n^*)^G$ of $\bigwedge N_n^*$ under the action of $G$.
% (this is interpreted as cohomology, but this is not important for us)  what is important is that 
By a result  of Chevalley transgression \cite{Chev} and  Dynkin \cite{dyn}
\begin{comment}
===

all references of cohomology can be extracted in a book  of which I give you the bibtex 
there he mentions several papers of Dynkin

 %@
  book {MR0400275,
    AUTHOR = {Greub, Werner and Halperin, Stephen and Vanstone, Ray},
     TITLE = {Connections, curvature, and cohomology},
      NOTE = {Volume III: Cohomology of principal bundles and homogeneous
              spaces,
              Pure and Applied Mathematics, Vol. 47-III},
 PUBLISHER = {Academic Press [Harcourt Brace Jovanovich Publishers]},
   ADDRESS = {New York},
      YEAR = {1976},
     PAGES = {xxi+593},
   MRCLASS = {58-02 (55FXX 57FXX 57D20)},
  MRNUMBER = {0400275 (53 \#4110)},
MRREVIEWER = {D. Lehmann},}

==

\end{comment}
		 this  is  the exterior algebra  in the elements 
%$\tr(S_{2h+1}(x_1,\ldots,x_{2h+1})),\ h=1,2,\ldots,n-1$. 
%This is explained by showing that $AT_n:=(\bigwedge\mathfrak g^*)^G$ is a Hopf algebra and the elements   $tr(S_{2h+1}(x_1,\ldots,x_{2h+1}))$ are its {\em primitive elements}.
%Now  it is important but I have only started to think about it, to analyze  the multilinear and antisymmetric  expressions in non commutative variables and their traces. From a purely formal point of view  they are described by a sort of curious symbolic calculus.
$$T_h:=\tr\bigl(S_{2h+1}(x_1,\ldots,x_{2h+1})\bigr),\quad 1\leq h\leq n-1.$$ 
  In this subsection we will deal with $AT_n:=(\bigwedge M_n^*)^G$ rather than with $(\bigwedge N_n^*)^G$. From this result it easily follows that, with a slight abuse of notation, the former is the exterior algebra in  the elements $T_0:=\tr(S_1(x_1)), T_1,\dots, T_{n-1}$.
We remark that we use only traces of the standard polynomials of odd degree since, as it is well-known, $\tr(S_{2h}(x_1,\ldots,x_{2h}))=0$ for every $h$, see  \cite{ros}. 
 
 The group $G$ obviously acts by automorphisms, thus $A_n=(\bigwedge M_n^*\otimes M_n)^G$ is indeed an associative algebra.
The main known fact that  we shall use is (see, e.g., \cite{K} or \cite[Corollary 4.2]{Reed}):

\begin{proposition}\label{dimn}
The dimension of $A_n$ over $F$ is $n2^n$.
\end{proposition}
%Now it is remarkable,   with this algebra product we have that $S_k=X^k$!! This in fact comes from the definitions but we can also see it in coordinates.
%Let us express  the product in this algebra  in  coordinates. We interpret an element $\phi_1\otimes\phi_2\otimes\ldots\otimes \phi_m\otimes A$, $\phi_i\in M_n^*,\ A\in M_n$, as a multilinear map.
%This gives
%$$(x_1,\ldots,x_m)\mapsto \prod\langle\phi_i\,|\,x_i\rangle A.$$
%In particular  the product $x_1\cdots x_m$ is the matrix whose $(h,k)$ entry is $$\sum_{i_1,i_2,\ldots,i_{m-1}} x^{(1)}_{h,i_1}x^{(2)}_{ i_1,i_2}\ldots x^{(m)}_{i_{m-1},k},$$ so if $x_{i,j}\in M_n^*$  denotes the coordinate in position $(i,j)$,  we see that the multiplication is given by
%$$x_1\cdots x_m=\sum_{h,k}\Bigl[ \sum_{i_1,i_2,\ldots,i_{m-1}} x_{h,i_1}\otimes x_{ i_1,i_2}\otimes \ldots\otimes x_{i_{m-1},k}\Bigr]e_{h,k}.$$
%This  implies that
% $$S_m =\sum_{h,k}\Bigl[ \sum_{i_1,i_2,\ldots,i_{m-1}} x_{h,i_1}\wedge x_{ i_1,i_2}\wedge \ldots\wedge x_{i_{m-1},k}\Bigr]e_{h,k},$$
 %and
%$$S_a\wedge S_b =\sum_{h,k}\sum_\ell \Bigl[ \sum_{i_1,i_2,\ldots,i_{a-1}} x_{h,i_1}\wedge x_{ i_1,i_2}\wedge \ldots\wedge x_{i_{a-1},\ell}\Bigr]\wedge  \Bigl[ \sum_{j_1,j_2,\ldots,j_{b-1}} x_{\ell,j_1}\wedge x_{ j_1,j_2}\wedge \ldots\wedge x_{j_{b-1},k}\Bigr]e_{h,k}.$$

Inside  $A_n$  we have the identity map $X$  which in the natural coordinates is  the {\em generic matrix} $\sum_{h,k}x_{hk}e_{hk} $.  By iterating the definition of wedge product we have the important fact (see also \cite{PA}):
\begin{proposition}
As a multilinear function, each power  $X^a:=X^{\wedge a}$  equals the standard polynomial $S_a$.
\end{proposition}
As a consequence we have $S_a\wedge S_b=S_{a+b} $ and by the Amitsur-Levitzki Theorem $X^{2n}=0$.
%\footnote{In fact one can easily prove this Theorem using this description of $S_{2n}$.}
 We summarize the rules:
 
$$ S_a=X^a,\,T_h\wedge X=-X\wedge T_h,\, X^{2n}=0,$$
where the powers of $X$ should be understood with respect to the wedge product. 
%\marginpar{actually the fact that $X^h$ is the standard polynomial is just a consequence of the definition of wedge product so it does not really need any proof}
\begin{remark}\label{baa}
Note that the elements 
$$T_{h_1}\wedge T_{h_2}\ldots \wedge T_{h_i}\wedge X^k,$$
where $h_1<h_2<\ldots <h_i$ and  $k$ is arbitrary,
 form  a linear basis of the  algebra of  multilinear and antisymmetric  expressions in noncommutative variables and their traces. 
\end{remark}
%\footnote{reference or explanation?}. 
We can consider this algebra as the exterior algebra in the variables $T_h$,  and a variable $X$ in degree 1 which anticommutes with the $T_i$. 
%Since $T_h=0$ on $M_n$ for $h\geq n$, let us consider the subalgebra that involves only $T_h$ for  $0\leq h\leq n-1$.
We now factor out the ideal of elements of degree $>n^2$ and $T_h$ for $h\geq n$, and thus obtain  a symbolic algebra  which we  call $\mathcal {TA}_n$. 
The algebra $A_n$ of multilinear antisymmetric invariant functions  on matrices to matrices is a quotient of this algebra. 
  We have to discover the  identities that generate the corresponding ideal, as for instance the Amitsur-Levitzki identity $X^{2n}=0$, which is the basic  {\em even} identity. % $X^{2n}=0$. 
 The next lemma points out the basic {\em odd} identity.
\begin{lemma}
The element  $O_n:= n X^{2n-1}-\sum_{i=0}^{n-1} X^{2i}\wedge T_{n-i-1 }\in \mathcal {TA}_n$ is an identity of $M_n$. Moreover, $O_n$ is an antisymmetric trace identity of minimal degree.
\end{lemma}
\begin{proof}
%Start from CH in polarized form   $CH:=\sum_{\sigma\in S_n}\epsilon_\sigma\phi_\sigma(x_1,\ldots, x_n)$. 
We know by  the SFT that every trace identity is obtained from $Q_n$ by substitution of variables and multiplication, hence any antisymmetric identity is obtained by first applying such a procedure obtaining a multilinear identity and then antisymmetrizing.  Since $Q_n$ is symmetric this procedure gives zero if we substitute  two variables by two monomials of the same odd degree. In particular this means that we can keep at most one variable unchanged and we have to substitute the others with monomials of degree $\geq 2$, thus the minimal identity that we can develop in this way is by substituting $x_2,\ldots,x_n$ with distinct monomials $M_i$, $2\leq i\leq n$, of degree 2 and then antisymmetrizing.

We use the formula \eqref{ch}.
If a permutation $\sigma$ contains a cycle  $(i_1,\ldots,i_k)$ in which neither 1 nor $n+1$ appear, substituting and alternating into the corresponding element $\phi_\sigma(x_1,\ldots,x_n)$, we get that the antisymmetrization of the factor $\tr(M_{i_1}\ldots M_{i_k})$ is zero  as $\tr(X^{2i})=\tr(S_{2i})=0$. 
Thus the only terms of \eqref{ch} which give a contribution are the ones where either $\sigma$  is a unique cycle and they contribute to  $(-1)^{n }n!X^{2n-1}$  or the ones with two cycles, one containing 1 and the other $n+1$;  such permutations can be described in the form
$$\sigma=(1,i_1,\ldots,i_h)(i_{h+1},\ldots, i_{n-1},n+1).$$ For each $h$ there are exactly $(n-1)!$ of these and they all have the sign $(-1)^{n-1}$.
The antisymmetrization of  $\phi_\sigma$ after substitution gives  $$\tr(X^{2h+1})X^{2(n-h-1)}=X^{2(n-h-1)}\wedge T_h,$$ and the claim follows.
\end{proof}
%\begin{example}
%$$n=1:\quad X-T(X),\quad n=2:\quad 2X^3-X^2\wedge T(X)-T(X^3)$$
%$$n=3:\quad 3X^5-X^4\wedge T(X)-X^2\wedge T(X^3)-T(X^5)$$
%\end{example} 

%Before studying the general case let us see the implications of these identities for $n=2$.
%We also have by multiplying by $X,X^2$ and using $X^4=0$
%$$0= X^3\wedge T(X)+X\wedge T(X^3)\implies [X^2\wedge T(X)+T(X^3)]\wedge T(X)+2X\wedge T(X^3)=0  $$
%$$\implies  T(X^3) \wedge T(X)+2X\wedge T(X^3)=0  \implies   X\wedge T(X)\wedge T(X^3)=0  $$ $$X^2\wedge T(X^3)=0 $$we have as alternating invariants 8 elements which we shall see form a basis:
%$$ 1,\quad X,\quad T(X),\quad X^2,\quad X\wedge T(X),\quad T(X^3),\quad X^2\wedge T(X),\quad   X\wedge T(X^3).$$

By $AT_{n-1}$ we denote the subalgebra of $AT_n$ generated by the $n-1$ elements $T_i$,  $0\leq i\leq n-2$. This  is an exterior algebra and has dimension $2^{n-1}$.

\begin{theorem}\label{thb} $A_n=(\bigwedge M_n^*\otimes M_n)^G$ is a free left module over the algebra  $AT_{n-1}$ with  basis  $\{1,X,\ldots,X^{2n-1}\}$.
The kernel of the canonical homomorphism from $\mathcal {TA}_n$ to $(\bigwedge M_n^*\otimes M_n)^G$ is generated by  $X^{2n}$ and $O_n$.
\end{theorem}
\begin{proof} We have that $\dim(\bigwedge M_n^*\otimes M_n)^G=n2^n$ by Proposition \ref{dimn}. Moreover, by Remark \ref{baa} and the FFT Theorem of invariant theory of matrices \ref{FST}, we know that $A_n$ as module over $AT_n$  is generated by the elements $1,X,\ldots,X^{2n-1}$.

 Now consider the left submodule $N$ of $(\bigwedge M_n^*\otimes M_n)^G$ generated over the algebra $AT_{n-1}$  by the elements $1,X,\ldots,X^{2n-1}$. Clearly $\dim(N)\leq(2n)2^{n-1}=n2^n$ and the equality holds if and only if $N$ is a free module. By the dimension formula  this is also equivalent to say that $N$ coincides with $(\bigwedge M_n^*\otimes M_n)^G$.  
 
  So it is enough to show that $N$ coincides with $(\bigwedge M_n^*\otimes M_n)^G$.
For this it suffices to show that $N$ is stable under multiplication by the {\em missing generator} $T_{n-1}$. Due to the commutation relations it  is enough to use the right multiplication, which is an  $AT_{n-1}-$linear map.

 From the identity $O_n$ we have  $$1\wedge T_{n-1}=T_{n-1}=-\sum_{i=1}^{n-1}X^{2i} \wedge  T_{n-i-1} +nX^{2n-1},$$  hence  for all $i\geq 1$ we have
  $$X^j\wedge T_{n-1}=-\sum_{i=1}^{n-[\frac j 2]}X^{2i+j} \wedge  T_{n-i-1},$$ which gives the matrix of such multiplication in this basis as desired.
  \end{proof}

\subsection{Antisymmetric identities that are not a consequence of the Cayley-Hamilton identity}  
We have denoted by $N_n$ the subspace of trace zero matrices and $G=PGL(n,F)$  acts on $M_n$ and $N_n$ by conjugation.
 
 We work with the associative algebra 
 $(\bigwedge N_n^*\otimes M_n)^G$  of $G$-equivariant antisymmetric  multilinear functions from  $N_n$ to $M_n$.  We let $Y$ be the  element  of $N_n^*\otimes M_n=\hom(N_n, M_n)$  corresponding to the inclusion. We note that $Y=X-\frac{\tr(X)}{n}.$
 
It easily follows from Theorem \ref{thb} and the previous formula that also $(\bigwedge N_n^*\otimes M_n)^G$ is a free module on the powers $Y^i$, $0\leq i\leq 2n-1$,  over the exterior algebra  in the $n-2$ generators   $\tr(Y^{2i+1})$, $1\leq i\leq n-2$,  
(note that $\tr(Y)=0$).

Finally, we know that the element $\tr(Y^{2n-1})$ acts on this basis by the basic formula:
  $$Y^j\wedge \tr(Y^{2n-1})=-\sum_{i=1}^{n-[\frac j 2]}Y^{2i+j} \wedge  \tr(Y^{2n-2i-1}).$$
We now construct the formal algebra of symbolic expressions by adding to $(\bigwedge N_n^*\otimes M_n)^G$ a variable $X$ with the rules
$$XY=-YX,\quad X\tr(Y^{2i+1})=-\tr(Y^{2i+1})X. $$  
We place $X$ in degree  1 and factor out all elements of degree $>n^2$. We call this formal algebra $\tilde A_n$. 
Its connection to quasi-identities will be revealed below.

Consider now the  algebra  $\mathbb F_n:=\bigwedge N _n^*[X]$  with again $X$ in degree 1, $X^{2n}=0$  and $X$ anticommutes with the elements of degree 1, that is with $N_n^*$. We also impose that  the expressions of degrees $>n^2$ are zero in $\mathbb F_n$.  Each element of this algebra induces an antisymmetric multilinear functions from  $N_n$ to $M_n$ and the elements that give rise to the zero function are exactly the antisymmetric multilinear quasi-identities on $N_n$.
As above,  we set    $T_i=\tr(Y^{2i+1}) \in \bigwedge^{2i+1}N _n^*$.
%\footnote{Should not be slightly better to write $T_i=\tr(X^{2i+1})$ in view of the beginning of page 20?}
Let us first identify  the subspace of $\mathbb F_n$  of quasi-identities deduced  from $Q_n$.
\begin{proposition}\label{pripo}
The space of quasi-identities deduced from $Q_n$ in $\mathbb F_n$   is the ideal generated by the element
$$O_n:= n X^{2n-1}-\sum_{i=0}^{n-2} X^{2i}\wedge T_{n-i-1}. $$
\end{proposition}
\begin{proof}
By definition a quasi-identity is deduced from $Q_n$   if it is obtained by first substituting the variables in $Q_n$ with monomials, and then multiplying by monomials and polynomials in the coordinates. If it is multilinear this procedure passes only through steps in which all substitutions are multilinear, as for antisymmetrizing we can first make it multilinear then antisymmetrize. Thus we see that the quasi-identities in $\mathbb F_n$ deduced from $Q_n$ equal the ideal generated by the invariant antisymmetric quasi-identities deduced from $Q_n$. By  Theorem \ref{thb} these are multiples of $O_n$,  proving the result. (Note that we have slightly abused the notation -- since we are dealing with trace zero matrices $N_n$ we have $T_0=0$, unlike in Theorem \ref{thb}.)
\end{proof}

We set $J:= O_n \mathbb F_n$ to be    the ideal generated by the element 
$O_n$.
We will concentrate on degree $n^2$ where we know that all formal expressions are identically zero  as functions on $N_n$. We want to describe in the space  of the quasi-identities  of degree $n^2$, $\mathbb F_n[n^2]$, the subspace $J\cap \mathbb F_n[n^2]$ of the elements which are a consequence of $Q_n$.
%This is the isotypic component of type  $N_n$ in $\mathbb F_n[n^2]\cap J$.  

%Since $O_n$ is homogeneous of degree  $2n-1$, we have 
%$$\mathbb F_n[n^2]\cap J= \bigwedge^{n^2-2n+1} N _n^*\,O_n.  
%$$

\subsubsection{Restricting to an isotypic component}
Let us notice that  the group $G$ acts on $\mathbb F_n$ through its action on $\bigwedge N_n^*$ and fixing $X$. Namely, we have a representation  
\begin{equation}
\label{deff}\mathbb F_n=\oplus_{i=0}^{2n-1}(\oplus_{j=0}^{n^2-i}\bigwedge^jN_n^*) X^i,\quad \mathbb F_n[n^2]=\oplus_{i=1}^{2n-1} \bigwedge^{n^2-i}N_n^* X^i.
\end{equation}
We  now restrict to the subspace  stable under $G$ and corresponding to the isotypic component  of  type $N_n$. % We want to see how to develop the calculus on this component.\smallskip
This is motivated by the fact that the component of $\mathbb F_n[n^2]$ relative to $X^2$ is $\bigwedge^{n^2
-2}N_n^* X^2$, which is visibly isomorphic to $N_n$ as a representation.  It is explicitly described as follows: the space $\bigwedge^{n^2-2}N_n^*  $ of multilinear antisymmetric functions of $n^2-2$ matrix variables  can be thought of as the span of the determinants of the maximal minors (of size $n^2-2$) of the $(n^2-2)\times (n^2-1)$ matrix whose $i^{th}$ row are the coordinates of the $i^{th}$  matrix variable $X_i$ which is assumed to be of trace 0.  
%%%%%%%%%%%%%%%%%
\begin{comment}
The space $\bigwedge^{n^2-2}N_n^*  $ as multilinear functions of $n^2-2$ matrix variables  can be thought of as the span of the determinants of the maximal minors (of size $n-2$) of the $n-2\times n-1$ matrix whose $i^{th}$ row consists of the coordinates of the $i^{th}$  matrix variable $X_i$.  
\end{comment}
%%%%%%%%%%%%%%%%%%%%%%%%%
%\begin{example}[ 4 variables]
% $$X_i=\begin{vmatrix}
%a_i&b_i\\
%c_i&-a_i
%\end{vmatrix} $$
% The determinants of 
%  $$ \begin{vmatrix}
%a_i&b_i&c_i\\
%a_j&b_j&c_j\end{vmatrix}\quad  (a_ib_j-a_jb_i), -(a_ic_j-a_jc_i), (b_ic_j-b_jc_i), $$ we choose one of these and antisymmetrize
%  $$ Ant((a_1b_2-a_2b_1) X_3X_4)$$+  
 % $$ (a_1b_2-a_2b_1) (X_3X_4 -X_4X_3)-(a_3b_2-a_2b_3) (X_1X_4 -X_4X_1)- (a_4b_2-a_2b_4) (X_3X_1 -X_1X_3)$$
 %$$ -(a_1b_4-a_4b_1) (X_3X_2 -X_2X_3)-(a_3b_4-a_2b_3) (X_1X_2 -X_2X_1)- (a_1b_3-a_3b_1) (X_2X_4 -X_4X_2)$$
%We have exactly 3 of these functional identities and they transform as the adjoint representation.
%\end{example} 

Let us denote by    $\mathbb G_n[n^2]$ the isotypic component  of  type $N_n$ in $\mathbb F_n[n^2]$, and by $\mathbb G_n[n^2]_{CH}$  the part of this component deducible from $Q_n$. We are now in a position to state our main result.

\begin{theorem}\label{main} We have a direct sum decomposition
$$\mathbb G_n[n^2]= \mathbb G_n[n^2]_{CH}\oplus \bigwedge^{n^2-2}N_n^* X^2.$$
\end{theorem}
In particular we have the following corollary.
\begin{corollary}\label{2}
The space $\bigwedge^{n^2-2}N_n^* X^2$ consists of quasi-identities which are not a consequence of the Cayley-Hamilton identity $Q_n$.
\end{corollary}

\begin{remark}
Corollary \ref{2} shows that there exist quasi-identities on $N_n$ which are not a consequence of the Cayley-Hamilton identity $Q_n$. However, by substituting the variable $x_k$ with $x_k-\frac{1}{n}\tr(x_k)$ we  readily obtain quasi-identities on $M_n$ that do not follow from $Q_n$. This corollary therefore answers our basic question posed in the introduction.
\end{remark}
The following question now presents itself.

 \begin{question}
 What is the minimal degree of a quasi-identity of $M_n$ that is not a consequence of the Cayley-Hamilton identity, and  how many variables it involves?
 \end{question}

Before engaging in the proof of  Theorem \ref{main} we need to develop some formalism.

First of all  recall that for a reductive group $G$,  a representation $U$, and an irreducible representation $N$,   we have a canonical isomorphism
$$\hom_G(N,U)=(N^*\otimes U)^G,\quad j:(N^*\otimes U)^G\otimes N\stackrel{\cong}\longrightarrow U_N;\ j[(\phi\otimes u)\otimes n]\mapsto \langle\phi\,|\,n\rangle u,$$
where $U_N$ denotes the isotypic component of type $N$.

We want to apply this isomorphism to $U= \mathbb F_n,$ or $\mathbb F_n[n^2]$ and $ N=N_n\cong N_n^*$. In particular we have to start describing  $(N_n\otimes \mathbb F_n)^G$. In fact it is necessary to work with
  \begin{equation}\label{tia}
\tilde A_n=(M_n\otimes \mathbb F_n)^G=((F\oplus N_n)\otimes \mathbb F_n)^G= \mathbb F_n^G\oplus (N_n\otimes \mathbb F_n)^G.
\end{equation}  We have
  \begin{equation}\label{An}
\tilde A_n=\mathbb F_n^G\oplus (N_n\otimes \mathbb F_n)^G=\oplus_{i=0}^{2n-1} \oplus_{j=0}^{n^2-i}(  \bigwedge^jN_n^*)^G X^i\oplus_{i=0}^{2n-1} \oplus_{j=0}^{n^2-i}(N_n\otimes \bigwedge^jN_n^*)^G X^i,
  \end{equation}
$$
 (M_n\otimes \mathbb F_n[n^2])^G=\mathbb F_n[n^2]^G\oplus (N_n\otimes \mathbb F_n[n^2])^G=  \oplus_{i=1}^{2n-1}  (  \bigwedge^{n^2-i}N_n^*)^G X^i  \oplus_{i=1}^{2n-1}   (N_n\otimes \bigwedge^{n^2-i}N_n^*)^G X^i.$$Now $\tilde A_n$ is still an algebra containing   $\mathbb F_n^G$ as a subalgebra. This is the algebra described at the beginning of this subsection, where $Y$ denoted the  generic trace zero matrix.
 \begin{lemma}
We have 
\begin{equation}\label{laff}
  (\mathbb F_n[n^2]\cap J )^G\oplus (N_n\otimes [\mathbb F_n[n^2]\cap J ])^G= (M_n\otimes [\mathbb F_n[n^2]\cap J ])^G= \tilde A_n[n^2]\cap   (1\otimes O_n)\tilde A_n  
\end{equation} 
and under the isomorphism $j: (N_n\otimes \mathbb F_n[n^2])^G\otimes N_n \to \mathbb G_n[n^2]$  the space $\mathbb G_n[n^2]_{CH}$  corresponds to $ (N_n\otimes [\mathbb F_n[n^2]\cap J ])^G\otimes N_n$.
\end{lemma}  
\begin{proof}
We have $ (M_n\otimes [\mathbb F_n[n^2]\cap J ])^G= \tilde A_n[n^2]\cap      (M_n\otimes   J  )^G $.  Since $O_n$ is $G$- invariant,  $O_n$  acts (by multiplication with $1\otimes O_n$)  on the space of invariants $ \tilde A_n =(M_n\otimes \mathbb F_n)^G$, that is $$  (M_n\otimes   J  )^G   =   (M_n\otimes O_n \mathbb F_n)^G=(1\otimes O_n ) (M_n\otimes \mathbb F_n)^G= (1\otimes O_n)\tilde A_n,$$
proving \eqref{laff}.

By definition $\mathbb G_n[n^2]_{CH}= \mathbb G_n[n^2]\cap J= \mathbb G_n[n^2]\cap O_n\mathbb F_n $. Since by definition $\mathbb G_n[n^2]$ is the isotypic component of type $N_n\cong N_n^*$ in $\mathbb F_n[n^2]$, we have $ (N_n\otimes  \mathbb F_n[n^2]  )^G= (N_n\otimes  \mathbb G_n[n^2]  )^G$. Thus clearly 
$$(N_n\otimes  \mathbb G_n[n^2]_{CH} )^G = (N_n\otimes [\mathbb G_n[n^2]\cap J ])^G=(N_n\otimes [\mathbb F_n[n^2]\cap J ])^G .$$ \smallskip
\end{proof}
  On the other hand,  the elements  $T_i\in ( \bigwedge^{2i+1}N_n^*)^G$ equal  $\tr(Y^{2i+1})$, so in particular 
 $T_{n-1}\in \mathbb F_n ^G$ acts on $\tilde A_n $ as  
$$T_{n-1}=n Y^{2n-1}-\sum_{i=1}^{n-2} Y^{2i}\wedge T_{n-i-1}.$$
Thus, we have 
 
\begin{lemma}
On $\tilde A_n $  the element $1\otimes O_n$ acts by multiplying  by $$\bar O_n:=n(X^{2n-1}- Y^{2n-1})-\sum_{i=1}^{n-2}( X^{2i}- Y^{2i})\wedge T_{n-i -1}   .$$
\end{lemma} 
Our goal is to understand $\mathbb G_n[n^2]_{CH}$. On the other hand,  $ \mathbb F_n[n^2] $  consists of all quasi-identities,  hence  $\mathbb F_n[n^2]  ^G$ is formed of trace identities and so it is contained in $J$. Thus, $ (\mathbb F_n[n^2]\cap J )^G=\mathbb F_n[n^2]  ^G$ and from  \eqref{laff}  we have
\begin{equation}
\label{uff} \tilde A_n[n^2]\cap   \bar O_n\tilde A_n= \mathbb F_n[n^2]  ^G\oplus (N_n\otimes \mathbb G_n[n^2]_{CH})^G.
\end{equation}

In order to study the isotypic component of type  $N_n$ in $\mathbb F_n[n^2]\cap J$ we therefore need to analyze 
\begin{equation}\label{top}(M_n\otimes [\mathbb F_n[n^2]\cap J])^G= \tilde A_n[n^2]\cap   \bar O_n\tilde A_n=\tilde A_n[n^2-2n+1]\bar O_n.\end{equation}
\begin{lemma}
We have a  {\em monomial basis} in $\tilde A_n $ made of elements of   the form $\mathcal T X^iY^j$ where $\mathcal T$ is a product of some of the elements $T_k$, $1\leq k\leq n-2$, written in the increasing order.  Its degree is $i+j$ plus the sum of the $2k+1$ for the $T_k$ appearing in $\mathcal T$. 
\end{lemma}
\begin{proof}
This follows from  \eqref{An} and Theorem \ref{thb}.
\end{proof}
  From \eqref{top} we  need to understand the monomials in degree  $n^2$ and $n^2-2n+1$ which are bases  of  $\tilde A_n,\ \tilde A_n[n^2-2n+1]$, respectively, and consider the matrix in these bases of multiplication by $\bar O_n$ as a map $$\pi_n: \tilde A_n[n^2-2n+1]\to \tilde A_n[n^2].$$
In order to understand the image of $\pi_n$ we construct a linear function $\rho$ on $\tilde A_n[n^2]$ defined on the monomial $M:=\mathcal T X^iY^j$ of degree $n^2$ as follows.
 \begin{enumerate}
\item If $\mathcal T$ does not contain at least two of the factors $T_h,T_k$, we set $\rho(M)=0$.
\item If $\mathcal T$ does not contain only one factor  $T_h$, we set $\rho(M)= (-1)^{h+n}$.
\item If $\mathcal T$   contains all the factors $T_k$ and $2\leq i,j\leq 2n-2$ are even, we set $\rho(M)= n $; otherwise we set $\rho(M)=0$.
\end{enumerate}
We  denote by $\mathcal S$ the ordered product of all $T_k,\ 1\leq k\leq n-2$, an element of degree $n^2-2n $.

%Then Theorem \ref{main}  follows from
\begin{proposition}\label{kerim}          
The image of $\pi_n$ equals the kernel of $\rho$. 
Moreover, 
$$\tilde A_n[n^2]=\im\, \pi_n\oplus F\mathcal SX^2Y^{2n-2}.$$
\end{proposition}

 \begin{proof}
First we prove that the  image of $\pi_n$ is contained in  the kernel of $\rho$.

For this take any monomial $A=\mathcal T X^iY^j\in \tilde A_n[n^2-2n+1] $ and consider  $A\bar O_n$.
\smallskip

i)\quad Firstly, if  $\mathcal T $  misses at least 3 of the elements $T_i$  then all the terms in  $A\bar O_n$  miss  at least 2 of the elements $T_i$, thus $\rho$ is 0 on all terms.
\smallskip

ii)\quad Assume    $\mathcal T $  misses two elements $T_h,T_k$.  The terms    $A\,n(X^{2n-1}- Y^{2n-1})$ in  $A\bar O_n$  then miss  at least 2 of the elements $T_i$, so $\rho$ is 0 on these terms.
 The  remaining nonzero terms come from $ B:=-[A ( X^{2(n-1-h)}- Y^{2(n-1-h)})\wedge T_{ h  } +A ( X^{2(n-1-k)}- Y^{2(n-1-k)})\wedge T_{k} ] .  $  Observe first that $\sum_{i=1}^{n-2}2i+1= n^2-2n  $ and so the degree of  $\mathcal T $  is 
 $n^2-2n -2(h+k)-2.$ The degree of $A$ is $n^2-2n+1$,  so that $i+j=   2h+2k+3.$ We may assume $h>k,\  i\geq j$.
 
 %Now we distinguish two cases.  
If $i+2(n-1-h)<2n$ (and hence $j+2 (n-1-h)<2n$),  then  $-A ( X^{2(n-1-h)}- Y^{2(n-1-h)})\wedge T_{ h }$  is the difference of two monomials on which  $\rho$ attains the same value, so on this expression $\rho$ vanishes, same for $k$.

%It remains to deal with the case  $i+2(n-1-h)\geq 2n$ or $j+2 (n-1-h)\geq 2n$. 
If $i+2(n-1-h)\geq 2n$, i.e., $i  \geq 2h+2$, and hence $j\leq 2k+1< 2h+1$, we have 
 $$ B=A Y^{2(n-1-h)} \wedge T_{ h  } +A  Y^{2(n-1-k)} \wedge T_{k}= A  \wedge T_{ h  }Y^{2(n-1-h)} +A \wedge T_{k}   Y^{2(n-1-k)} $$
 $$=(-1)^{i+j} (\mathcal T  \wedge T_{ h  }X^i Y^{2(n-1-h)+j} +  \mathcal T  \wedge T_{ k  }X^i Y^{2(n-1-k)+j} ). $$
When we place $ \mathcal T  \wedge T_{ h  }$ in the increasing order  we multiply by $(-1)^u$ where $u$ is the number of factors  of index $>h$. Since we have $n-2$  factors,  $(-1)^u= (-1)^{n-2-h} $ and the value of $\rho$ on the first term  is $ (-1)^{i+j}(-1)^u(-1)^{k+n}= (-1)^{i+j+h+k} $.  For the second term the number of terms we have to exchange is  the number of terms of index bigger than $k$  minus 1 so we get the sign $-(-1)^{i+j+h+k} $  and the two terms cancel.\smallskip

%Assume now that $j+2 (n-1-h)\geq 2n$. Then since we are assuming $i\geq j$ and $h>k$ we have $B=0$.
%The expression  $A ( X^{2(n-1-h)}- Y^{2(n-1-h)})\wedge T_{ h  }$ is the difference of two monomials  and is treated as before. We have $ i\geq j\geq 2h+2>2k+1$, hence   the remaining expression $A ( X^{2(n-1-k)}- Y^{2(n-1-k)})\wedge T_{k} =0. $   
\smallskip

iii)\quad Assume    $\mathcal T $  misses only one element  $T_h$. In this case the degree of $\mathcal T$ is   $ n^2-2 n -2 h -1 $, thus $i+j=   2h +2.$ The two terms $A\, n(X^{2n-1}- Y^{2n-1}) $  are 0 unless either $i=0$ or $j=0$, since we are assuming $i\geq j$ this implies $j=0,i=2h+2$ and \begin{equation}
\label{pric}A\,n(X^{2n-1}- Y^{2n-1}) =-n\mathcal T X^{2h+2}Y^{2n-1} .
\end{equation}  The other contribution to the product is 
$-A ( X^{2(n-1-h)}- Y^{2(n-1-h)})\wedge T_{ h  } $.    

  If we have $2(n-1-h)+i< 2n$, then on this contribution $\rho$ vanishes. In this case $i\neq 2h+2$, so the contribution \eqref{pric} does not appear.

If   $2(n-1-h)+i\geq  2n$, i.e.  $i\geq 2h+2$,  we have   $i=2h+2,j=0$. The product $A\bar O_n$ is 0 unless $2h+2<2n$, in this case  it equals
\begin{equation}
\label{unc}- n\mathcal T X^{2h+2}Y^{2n-1}  +(-1)^{n-h}\mathcal S X^{2h+2}Y^{2(n-1-h)}, 
\end{equation} 
where as before  we have used 
$\mathcal T \wedge T_{ h  }=(-1)^{n-h} \mathcal S$. % where  $\mathcal S$ is the ordered product of all the $T_i$.
By definition the value of $\rho$ on $\mathcal T X^{2h+2}Y^{2n-1} $ is $ (-1)^{n+h}$,  as for $\mathcal S X^{2h+2}Y^{2(n-1-h)} $  we have $2n+2\geq 2$  and also $2(n-1-h)\geq 2$,  thus the value of $\rho$ on it equals $ n $.  The value  on the sum  is therefore
$-n(-1)^{n+h}  +(-1)^{n-h}n =0.$
\smallskip

iv)\quad Finally we consider the case in which $\mathcal T$ is the product of all the $T_i$'s.  In this case $i+j=1 $, and since we assume $i\geq j$ we   have $i=1,j=0$.
The only possible terms in the product are in $A\,n(X^{2n-1}- Y^{2n-1})=- n\mathcal TXY^{2n-1}  $. By definition $\rho$ is $0$  on this term.
\smallskip

We now want to prove that the image of $\pi_n$ coincides with the kernel of $\rho$. For this we have to show that  the image of $\pi_n$ has codimension 1. It is enough to show that adding a single vector to the image of $\pi_n$ we obtain the entire space. 
We  define $V$ to be the space spanned by $\mathcal S X^{2} Y^{2n-2}$ and $\im(\pi_n) $.  We want to show that $V= \tilde A_n[n^2] $. 
In the case iv) we have already seen that $\mathcal SXY^{2n-1},\mathcal SX^{2n-1}Y$ belong to the image of $\pi_n$. 
\smallskip

{\bf Claim 1.} {\em 
For every $h$ we have $\mathcal SX^{2h+1}Y^{2(n-h)-1}\in \im(\pi_n),\; \mathcal SX^{2h }Y^{2(n-h) }\in V.$}

To prove this claim,
consider $T_h$ of degree $2h+1$.  We may remove $T_h$ from $\mathcal S$ obtaining a product $\mathcal S^{(h)}$ and take the element $$A:=\mathcal S^{(h)} X^{2h+1} Y \in \tilde A_n[n^2-2n+1].$$
We have $$A\bar O_n=\pm \mathcal S  X^{2h+1} Y(X^{2(n-h-1)}-  Y^{2(n-h-1) }) =\pm (\mathcal S  X^{2n-1} Y  -  \mathcal S  X^{2h+1} Y^{2 n-2h-1  }).$$  Since $\mathcal S  X^{2n-1} Y \in  \im(\pi_n)$ we deduce $  \mathcal S  X^{2h+1} Y^{2 n-2h-1  } \in  \im(\pi_n) $.\smallskip

For the other case consider  $A:=\mathcal S^{(n-h)} X^{2} Y^{2(n-h)} \in \tilde A_n[n^2-2n+1]$. Then
$$A\bar O_n=\pm \mathcal S  X^{2} Y^{2(n-h)}(X^{2(h-1)}-  Y^{2(h-1) }) =\pm (\mathcal S  X^{2h} Y^{2(n-h)}  -  \mathcal S  X^{2} Y^{2 n-2}).$$  
Since $\mathcal S  X^{2} Y^{2n-2} \in  V$ we have $  \mathcal S  X^{2h} Y^{2 (n-h)} \in  V $.
\smallskip

{\bf Claim 2.} {\em 
If $\mathcal T  X^iY^j\in V$, also $\mathcal T  X^jY^i\in V$.}

By definition $\im(\pi_n)$  is invariant under the exchange of $X,Y$,  while $V$ is obtained from $\im(\pi_n)$ by adding $\mathcal S X^{2} Y^{2n-2}$, but by Claim 1 we also have $\mathcal S X^{2n-2}Y^{2} \in V$.  This proves Claim 2.
\smallskip

{\bf Claim 3.} {\em 
All monomials  $\mathcal TX^iY^j\in  \tilde A_n[n^2 ]$, where $\mathcal T$ misses one  element $T_h$,  are in $V$.
}

 We must have $i+j=2(n+h)+1$. 
Apply  \eqref{unc}   and Claim 1 to deduce that $\mathcal T X^{2h+2} Y^{2n-1}\in V$. By Claim 2  also  $\mathcal T X^{2n-1} Y^{2h+2}$ belongs to $V$.
 We may  assume $i\geq j$ by Claim 2. 
It  thus suffices to consider only the case $i>2h+2$. 
Note that in the case $h=n-2$, we have $i+j=4n-3$, thus $i=2n-1,j=2n-2$, so in this case the previous argument establishes the claim.

Consider now the case $h<n-2$. 
We first consider the case $i=2n-2$. (Note that the case $i=2n-1$ has been considered above.) Take $A:=\mathcal T^{(n-2)}X^{2n-2}Y^{2h+1}$, where $\mathcal T^{(n-2)}$  denotes the element
obtained from $\mathcal T$ by removing $T_{n-2}$. Then 
$$A\bar O_n=\pm \mathcal TX^{2n-2}Y^{2h+3}\pm S^{(n-2)}X^{2n-2}Y^{2n-1}.$$
As $S^{(n-2)}X^{2n-2}Y^{2n-1}$ has already been proven to belong to $V$, $TX^{2n-2}Y^{2h+3}\in V$. 
We now prove by the decreasing induction  that  $\mathcal T X^{i}Y^{2(n+h)+1-i}$ lies in $V$ for $ i>2h+2$.  
Take $A:=\mathcal T^{(n-2)} X^{i}Y^{2h+2n-i -1}\in  \tilde A_n[n^2-2n+1]$. 
We have
\begin{align*} A\bar O_n=&\mathcal T^{(n-2)} X^{i}Y^{2h+2n-i-1 } \bar O_n \\
 =&\mathcal T^{(n-2)} X^{i}Y^{2h+2n-i-1 } [-(X^2-Y^2)\wedge T_{n-2}-( X^{2(n-1-h)}- Y^{2(n-1-h)})\wedge T_{ h  }]\\
=&
\pm\mathcal T    ( X^{i+2}Y^{2h+2n-i-1 }- X^{i}Y^{2h+2 n -i+1})\in V\end{align*}
in case $i<2n-2$. 
%In the cases $ i=2n-2, i=2n-1$ we get additional summands of the form $\mathcal S^{(n-2)} X^{2n-1} Y^{2n-2}$, $\mathcal S^{(n-2)} X^{2n-2} Y^{2n-1}$, respectively, which have already been proved to belong to $V$.
Since by the induction hypothesis $\mathcal T X^{i+2}Y^{2h+2n-i-1 } \in V$, it follows that  $\mathcal T X^{i}Y^{2h+2n-i+1 }\in V$, and we have the desired result. 

\smallskip

{\bf Claim 4.} {\em 
All monomials  $\mathcal TX^iY^j\in  \tilde A_n[n^2 ]$, where $\mathcal T$ miss $m\geq  2$  elements $T_h$,  are in $V$.
}

Assume that $\mathcal T$ misses elements $T_{h_1},\dots,T_{h_m}$. Let us denote $s=\sum_{i=1}^m (2h_i+1)$. %Note that $s\geq 8$.
We first show that $\mathcal TX^{2n-1}Y^{s+1}\in \im(\pi_n)$,$\mathcal TX^{2n-2}Y^{s+2}\in \im(\pi_n)$. 
Since $m\geq 2$ and $s\leq 2(2n-1)-2n=2n-2$, 
$\mathcal T$ cannot miss $\mathcal T_{n-2}$. 
Denote by $\mathcal T^{(n-2)}$ the element
obtained from $\mathcal T$ by removing $T_{n-2}$. 
As all monomials in $\mathcal T^{(n-2)}\bar O_n$ miss at least two elements $T_{k_1}$, $T_{k_2}$, they cannot miss $T_{n-2}$ by the previous argument, thus we have $\mathcal T^{(n-2)}X^{2n-1}Y^{s-1}\bar O_n=\pm \mathcal TX^{2n-1}Y^{s+1}$.
%since $s-1+2n-1-2h_i-1>2n$ implies $Y^{s-1+2n-1-h_i}=0$.
The same argument shows that $\mathcal TX^{2n-2}Y^{s+2}\in \im(\pi_n)$. 
Arguing by the decreasing induction we may assume that $\mathcal TX^{2n-k+2}Y^{s+k-2}\in \im(\pi_n)$. 
We have 
$$\mathcal T^{(n-2)}X^{2n-k}Y^{s+k-2}\bar O_n = \pm(\mathcal TX^{2n-k+2}Y^{s+k-2}-\mathcal TX^{2n-k}Y^{s+k})\in \im(\pi_n).$$
By the induction hypothesis $\mathcal TX^{2n-k+2}Y^{s+k-2}\in \im(\pi_n)$ and thus $\mathcal TX^{2n-k}Y^{s+k}\in \im(\pi_n)$.
\end{proof}

\begin{proof}[Proof of Theorem \ref{main}]
Note that $\mathcal SX^2Y^{2n-2}$, which is not in the image of $\pi_n$ by Proposition \ref{kerim}, is a generator of the 1-dimensional space  $(N_n\otimes\bigwedge^{n^2-2}N_n^* X^2)^G\subset(N_n\otimes \mathbb F_n[n^2])^G$. The decomposition of $A_n[n^2]$ from  Proposition \ref{kerim} thus induces the decomposition $\mathbb G_n[n^2]= \mathbb G_n[n^2]_{CH}\oplus \bigwedge^{n^2-2}N_n^* X^2.$
\end{proof}

\section{Quasi-identities that follow from the Cayley-Hamilton identity} \label{sec4}
In this final section we collect several  results on quasi-identities of matrices and the Cayley-Hamilton identity, and give a positive solution for the Specht problem for quasi-identities of matrices.

\subsection{Quasi-identities and local linear dependence}
\label{sec5}

Let $\mathfrak R$ be an $F$-algebra. Noncommutative polynomials $f_1,\ldots,f_t\in F\langle x_1,\ldots,x_m\rangle$ are said to be {\em $\mathfrak R$-locally linearly dependent} if the elements $f_1(r_1,\ldots,r_m),\ldots, f_t(r_1,\ldots,r_m)$ are linearly dependent in $\mathfrak R$ for all $r_1,\ldots,r_m\in \mathfrak R$.
This concept has actually appeared in Operator Theory  \cite{hel}, and was recently studied from the algebraic point of view in \cite{BK}. We will see that it can be used in the study of quasi-identities.

Recall that $C_m$ stands for the Capelli  polynomial.  The following well-known result 
(see, e.g., \cite[Theorem 7.6.16]{Row}) was used in \cite{BK} as a basic tool.

\begin{theorem}\label{Raz}  Let $\mathfrak R$ be a prime algebra. Then $a_1,\ldots,a_t\in\mathfrak R$ are linearly dependent over the extended centroid of $\mathfrak R$
if and only if  $C_{2t-1}(a_1,...,a_t,r_{1},...,r_{t-1})=0$ for all $r_1,\ldots,r_{t-1}\in\mathfrak R$.
\end{theorem}

By using a similar approach as in the proof of \cite[Theorem 3.1]{BK}, just by applying Theorem \ref{Raz} to the algebra of generic matrices instead of to the free algebra $F\langle X \rangle$, we get the following characterization of $M_n$-local linear dependence through the central polynomials (cf. \S \ref{cepo}).  \begin{theorem}\label{loc}
Noncommutative polynomials  $f_1,\ldots,f_t$ are $M_n$-locally linearly dependent if and only if there exist central polynomials 
$c_1,\dots,c_t$, not all polynomial identities, such that $\sum_{i=1}^t c_i f_i$ is a polynomial identity of $M_n$.
\end{theorem}

\begin{proof}
By Theorem \ref{Raz},
the condition that $f_1,\dots,f_t$ are $M_n$-locally linearly dependent is equivalent to the condition that  
$$H:=C_{2t-1}(f_1,\dots,f_t,y_1,\dots,y_{t-1})$$ 
is a polynomial identity of $M_n$.
Since $M_n$ and the algebra $F\langle \xi_k \rangle $ of $n\times n$ generic matrices satisfy the same polynomial identities, this is the same as saying that $H$ is a polynomial identity of $F\langle \xi_k \rangle $.
Using Theorem \ref{Raz} once again we see that this is further equivalent to the condition that
$f_1,\ldots,f_t$, viewed as elements of $F\langle \xi_k \rangle $, are linearly dependent
over the extended centroid of $F\langle \xi_k \rangle $. Since $F\langle \xi_k \rangle $ is a prime PI-algebra, its extended centroid is the field of fractions of the center of $F\langle \xi_k \rangle $; the latter can be identified with central polynomials, and hence the desired conclusion follows.  
\end{proof}

\begin{corollary}\label{din1}
If noncommutative polynomials  $f_0,f_1,\ldots,f_t$ are $M_n$-locally linearly dependent, while $f_1,\ldots,f_t$ are $M_n$-locally linearly independent, then there exist central polynomials 
$c_0,c_1,\dots,c_t$,  such that $c_0$ is nontrivial and  $\sum_{i=0}^t c_i f_i$ is a polynomial identity of $M_n$.
\end{corollary}

Later, in  Remark \ref{rem55}, we will show that this result can be used to obtain an alternative proof of a  somewhat weaker version of  Theorem \ref{cp2}. 

\begin{lemma}\label{novanova}
If a quasi-polynomial $\sum_{i=1}^t \lambda_i M_i$ is a quasi-identity of $M_n$, then either each $\lambda_i =0$ or  $M_1,\ldots,M_t$ are $M_n$-locally linearly dependent (and hence satisfy the conclusion of Theorem \ref{loc}).
\end{lemma}

\begin{proof}
We may assume that $\lambda_i=\lambda_i(x_1,\ldots,x_m)$ and $M_i = M_i(x_1,\ldots,x_m)$. The set  $W$ of all $m$-tuples 
 $(A_1,\dots,A_{m})\in M_n^{m}$ such that
 $\lambda_i(A_1,\dots,A_{m}) =0$ for  all  $ i=1,\ldots,  t$ is closed in the Zariski topology of $F^{n^2m}$.
 Similarly, the
 set $Z$ of  all $m$-tuples $(A_1,\dots,A_m)\in M_n^{m}$ such that $M_1(A_1,\ldots,A_{m}),\ldots,M_t(A_1,\ldots,A_{m})$  are linearly dependent is also closed - namely, the linear dependence can be expressed through zeros of a polynomial by Theorem \ref{Raz}. If $Z = M_n^m$, then  $M_1,\ldots,M_t$ are $M_n$-locally linearly dependent. Assume therefore that $Z \ne M_n^m$.
 Suppose that $W\ne M_n^{m}$. Then,
 since $F^{n^2m}$ is irreducible (as char$(F)=0$), the complements of $W$ and $Z$ in $M_n^{m}$  have a nonempty intersection. 
This  means that  
  there exist $A_1,\ldots,A_{m}\in M_n^m$ such that  $\lambda_i(A_1,\dots,A_{m}) \ne 0$ for some $i$  and $M_1(A_1,\ldots,A_{m}),\ldots,M_t(A_1,\ldots,A_{m})$ are linearly independent. However, this is impossible since  $\sum_{i=1}^t \lambda_i M_i$ is a quasi-identity.  Thus, $W=M_n^{m}$, i.e., each $\lambda_i=0$.
\end{proof}
We conclude this  subsection with a theorem giving a condition under which a quasi-identity is 
 a consequence of the Cayley-Hamilton identity.

\begin{theorem}\label{the56}
Let 
$P= \lambda_0 M_0  + \sum_{i=1}^t \lambda_i M_i \in \mathfrak{I}_n$.  If $M_1,\ldots,M_t$ are $M_n$-locally linearly independent, then $P$ is a consequence  of the Cayley-Hamilton identity.
\end{theorem}

\begin{proof} We may assume that some $\lambda_i\ne 0$, and so $M_0,M_1,\ldots,M_t$  are $M_n$-locally linearly dependent by Lemma \ref{novanova}.  Theorem \ref{loc} tells us that there exist central polynomials $c_0,c_i$, not all trivial, such that
$ c_0M_0 + \sum_{i=1}^t c_i M_i$ is a polynomial identity. Multiplying this identity with $\lambda_0$ and then comparing it with the quasi-identity $c_0 P$ it follows that $\sum_{i=1}^t (c_0\lambda_i - c_i\lambda_0)M_i  \in \mathfrak{I}_n$.   Lemma \ref{novanova} implies that $c_0\lambda_i = c_i\lambda_0$ for every $i$.  Let us write $\lambda_i=\lambda_i^G\lambda_i'$ where $\lambda_i^G$ is the product of all irreducible
factors of $\lambda_i$ that are invariant under the conjugation by $G=GL_n(F)$, and $\lambda_i'$ is the
product of the remaining irreducible factors of $\lambda_i$. Hence $c_i\lambda_0^G\lambda_0'=c_0\lambda_i^G\lambda_i'$ and therefore $\lambda_0'=\lambda_i'$ for every $i$. We
thus have $$P=\lambda_0'\bigl(\lambda_0^G M+\sum_{i=1}^t \lambda_i^G M_i\bigr).$$ Since $\lambda_i^G$ are
invariant under $G$, they are trace polynomials. Therefore the desired conclusion follows from the SFT.
\end{proof}

\begin{remark}\label{rem55}
Theorem \ref{cp2} in particular tells us that for  every quasi-identity $P$  of $M_n$ there exists a nontrivial central polynomial $c$ with zero constant term  such that $cP$ is a consequence  of the Cayley-Hamilton identity. Let us give an alternative proof of that, based on local linear dependence and the SFT.

We first remark that the condition that $c\in F\langle X \rangle$ is a central polynomial can be expressed as that there exists $\alpha_c\in \mathcal C$ such that $c - \alpha_c\in \mathfrak{I}_n$. Actually,  $c - \alpha_c$ is a trace identity since $\alpha_c = \frac{1}{n}{\rm tr}(c)$. Therefore the SFT  implies that 
for every central polynomial c of $M_n$ there exists $\alpha_c\in\mathcal C$ such that $c - \alpha_c$ is a quasi-identity of $M_n$ contained in the T-ideal of $ \mathcal{C}\bigl\langle X \bigr\rangle $ generated by $Q_n$.

Now take an arbitary $P\in\mathfrak{I}_n$, and let us prove that $c$ with the aforementioned property exists. 
Write 
$
P = \sum_{i=1}^{n^2} \lambda_{i} x_i+ \sum \lambda_M M
$
where each $M$ in the second summation is different from $x_1,\ldots,x_{n^2}$. We proceed by induction on the number of summands $d$ in the second summation. 
If $d=0$, then $P=0$ by Lemma \ref{novanova} (since $x_1,\ldots,x_{n^2}$ are $M_n$-locally linearly independent).
Let $d > 0$. Pick $M_0$ such that $M_0\notin \{x_1,\ldots,x_{n^2}\}$ and $\lambda_{M_0}\ne 0$. Note that $M_0,x_1,\ldots,x_{n^2}$ are $M_n$-locally linearly dependent, while $x_1,\ldots,x_{n^2}$ are $M_n$-locally linearly independent. Thus, by Corollary \ref{din1} there exist central polynomials $c_0,c_1,\ldots,c_{n^2}$ such that $c_0$ is nontrivial and 
$f:=c_0M_0 + \sum_{i=1}^{n^2} c_ix_i$ is an identity of $M_n$. Let $\alpha_i \in\mathcal C$, 
$i=0,1,\ldots,n^2$, be such that $c_i - \alpha_i$ is a quasi-identity lying in the T-ideal  generated by $Q_n$.
Let us define
$P':=\alpha_0 P-\lambda_{M_0} \alpha_0 M_0 -\lambda_{M_0}\sum_{i=1}^{n^2} \alpha_ix_i.$
Writing each $\alpha_i$ as $c_i - (c_i - \alpha_i)$ we see that $P'$ is a quasi-identity. Note that $P'$ involves $d-1$ summands not lying in $\mathcal C x_i$, $i=1,\ldots,n^2$. Therefore the induction assumption yields the existence of a nonzero central polynomial $c'$ such that $c'P'$ lies in the T-ideal generated by $Q_n$. Setting $c= c_0c'$  we thus have $c\ne 0$ and
\begin{align*}
cP =& (c_0 - \alpha_0)c'P + \alpha_0c'P
 = (c_0 - \alpha_0)c'P + c'P' + \lambda_{M_0}c'\bigl(\alpha_0 M_0 + \sum_{i=1}^{n^2} \alpha_ix_i\bigr)\\
 =& (c_0 - \alpha_0)c'P + c'P' - \lambda_{M_0}c'\bigl((c_0-\alpha_0) M_0 + \sum_{i=1}^{n^2} (c_i-\alpha_i)x_i\bigr) + \lambda_{M_0}c'f.
\end{align*}
The T-ideal generated by $Q_n$ contains $c_i-\alpha_i$, $0\le i\le n^2$, $c'P'$, as well as $f$ according to the SFT. Hence it also contains $cP$. 
\end{remark}

\subsection{Some special cases} The purpose of this subsection is to examine several simple situations, which should in particular serve as an evidence of the delicacy of the problem of finding quasi-identities that are not a consequence of the Cayley-Hamilton identity. 

We begin with quasi-polynomials $p=p(x)$ of one variable, i.e., 
$
p(x) = \sum_{i=0}^m \lambda_i(x)x^i$. Here we could  refer to results on more general functional identities of one variable in \cite{BS}, but  an independent treatment is very simple.

% We first record a lemma without a proof since it is practically the same as that of Lemma \ref{nova}. The crucial part is observing  that the set of all $A\in M_n$ such that $I,A,\ldots,A^{n-1}$ are linearly dependent is, as follows from Theorem \ref{Raz}, closed  with respect to the Zariski topology. On the other hand, this lemma also follows from \cite[Lemma 14.7]{Salt}.

\begin{proposition}\label{oneind}
If a quasi-polynomial of one variable $p(x)$ is a quasi-identity of $M_n$, then
there exists a quasi-polynomial $r(x)$ such that $p(x) =r(x) q_n(x)$.
\end{proposition}

\begin{proof}
Let $
p(x) = \sum_{i=0}^m \lambda_i(x)x^i$. The proof is by induction on $m$. Since $1,x,\ldots,x^{n-1}$ are $M_n$-locally linearly independent, 
we may assume that $m\ge n$ by Lemma \ref{novanova}. Note that $p(x)-\lambda_m(x)x^{m-n}q_n(x)$
is a quasi-identity for which the induction assumption is applicable. Therefore $p(x)-\lambda_m(x)x^{m-n}q_n(x) =
  r_1(x)q_n(x)$
for some $r_1(x)$, and hence   $p(x)= \bigl(\lambda_m(x)x^{m-n}+
  r_1(x)\bigr)q_n(x)$.
  \end{proof}

%This proposition implies that $q_n(h(x))$, where $h(x)$ is any quasi-polynomial  in one indeterminate,  can be written as $r(x) q_n(x)$ for some $r(x)$. Of course, there are other ways to establish this, but this proof is very short at any rate.

What about quasi-polynomials of two variables? At least for $2\times 2$ matrices, the answer comes easily.

\begin{proposition}\label{twoind}
If a  quasi-polynomial of two variables  $P=P(x,y)$ is a quasi-identity of $M_2$, then $P$ is a consequence of the Cayley-Hamilton identity.
\end{proposition}

\begin{proof}
It is an easy exercise to show that any quasi-polynomial of two variables $P=P(x,y)$  can be written as $P = \lambda_0 + \lambda_1 x + \lambda_2 y + \lambda_3 xy + R$ where $R$ lies in the T-ideal   generated by $Q_2$.  Thus, if $P$ is a quasi-identity, then each $\lambda_i=0$ by Lemma \ref{novanova}, so that $P=R$ is a consequence of the Cayley-Hamilton identity.  
\end{proof}

Now one may wonder what should be  the degree  of a quasi-identity that may not follow from the Cayley-Hamilton identity.  We first record a simple result which is a byproduct of the general theory of functional identities.

\begin{proposition}\label{threeind}
If $\sum\lambda_M M$ is a quasi-identity of $M_n$ and $\deg(\lambda_M) + \deg(M) < n$ for every $M$, then each $\lambda_M=0$.
\end{proposition}

\begin{proof}
Apply, for example, \cite[Corollary 2.23, Lemma 4.4]{FIbook}.
\end{proof}

The  question arises  what can be said about quasi-identities of degree $n$.
The multilinearization process works for the quasi-polynomials just as it works for the ordinary noncommutative polynomials.  The multilinear quasi-polynomials therefore deserve a special attention. By saying that 
 $P=P(x_1,\ldots,x_n)$ is  {\em multilinear} of degree $n$ we mean, of course, that $P$ consists of summands of the form $\lambda(x_{i_1},\ldots,x_{i_k}) x_{i_{k+1}}\cdots x_{i_{n}} $ where
$\{1,\ldots,n\}$ is the disjoint union of $\{i_1,\ldots,i_k\}$ and $\{i_{k+1},\ldots,i_m\}$, and   $\lambda(x_{i_1},\ldots,x_{i_k})\in \mathcal C$ is multilinear, i.e., it is a linear combination of monomials of the form $x_{s_1t_1}^{(i_1)}x_{s_2t_2}^{(i_2)}\cdots x_{s_kt_k}^{(i_k)}$. A basic example is the Cayley-Hamilton polynomial $Q_n$.

\begin{theorem}\label{multil}
Every multilinear quasi-identity of $M_n$ of degree $n$ is a scalar multiple of $Q_n$.
\end{theorem}

\begin{proof} Let  $S_{n,k}$, $1\le k\le n$, 
 denote the set of all permutations $\sigma\in S_n$ such that $\sigma(1)<\sigma(2)<\cdots<\sigma(k)$. For convenience we also set $S_{n,0}=S_n$.
  Note that a multilinear quasi-polynomial $P$ of degree $n$ can be written as
$$P(x_1,\dots,x_n)=\sum_{k=0}^{n}\sum_{\sigma\in S_{n,k}} \lambda_{k\sigma}(x_{\sigma(1)},\dots,x_{\sigma(k)})x_{\sigma(k+1)}\cdots x_{\sigma(n)}$$
(here, $\lambda_{0\sigma}$ are scalars). By $e_{ij}$ we denote the matrix units in $M_n$.

 We assume that $P$ is a quasi-identity, and proceed by a series of claims. 

\smallskip
{\bf Claim 1}. {\em  For all $\sigma\in S_{n,k}$, $1\leq k\leq n$, and all distinct $1\leq i_1,\dots,i_k\leq n$,    we have 
$$\lambda_{k\sigma}(e_{11},\dots,e_{kk})=\lambda_{k\sigma}(e_{i_1i_1},\dots,e_{i_ki_k}) =-\lambda_{k-1,\sigma}(e_{11},\dots,e_{k-1,k-1}).$$}

The proof is by induction on $k$.  First, take  $0\leq j\leq n-1$ and substitute  
$$e_{1+j,1+j},e_{1+j,2+j},e_{2+j,3+j},\dots,e_{n-1+j,n+j}$$
  (with addition modulo $n$) 
for  $x_{\sigma(1)},\dots,x_{\sigma(n)}$ in $P$. Considering the coefficient at $e_{1+j,n+j}$ we get
 $\lambda_{0\sigma}+\lambda_{1\sigma}(e_{1+j,1+j})=0$. This implies the truth of Claim 1 for $k=1$.  Let $k > 1$ and take $\sigma \in S_{n,k}$. 
Choose  a subset of $\{1,\dots,n\}$ with $k-1$ elements,  $\{i_{n-k+2},\dots,i_n\}$,  and let  $\{j_1,\dots,j_{n-k+1}\}$ be its complement.
Let us substitute
$$
e_{i_{n-k+2},i_{n-k+2}},\dots,e_{i_n,i_n},\,\,
e_{{j_1},j_1},e_{j_1,j_2},
e_{j_2,j_3},\dots,e_{j_{n-k},j_{n-k+1}}
$$
for $x_{\sigma(1)},\dots,x_{\sigma(n)}$, respectively, in $P$.  Similarly as above, this time by considering the coefficient at $e_{j_1,j_{n-k+1}}$,  we obtain 
$$
\lambda_{k-1,\sigma}(e_{i_{n-k+2},i_{n-k+2}},\dots,e_{i_n,i_n})+
\lambda_{k\sigma}(e_{i_{n-k+2},i_{n-k+2}},\dots,e_{i_n,i_n},e_{{j_1},j_1})=0.
$$
The desired conclusion follows from the  induction hypothesis.

\smallskip
{\bf Claim 2}. {\em  For all $\sigma,\tau \in S_{n,k}$, $0\leq k\leq n-1$, and all distinct $1\leq i_1,\dots,i_k\leq n$,    we have 
 $$\lambda_{k\sigma}(e_{11},\dots,e_{kk})=\lambda_{k\tau}(e_{i_1i_1},\dots,e_{i_ki_k}).$$}
%, for arbitrary distinct $1\leq i_k\leq n$, all $1\leq k\leq n-1$ and all
%$\sigma,\tau\in S_n(k)$.\smallskip
% We proceed by ''backward induction``. 
Evaluating $P$ at $e_{11},\dots,e_{nn}$ results in 
$$\lambda_{n-1,\sigma_i}(e_{11},\dots,e_{i-1,i-1},e_{i+1,i+1},\dots,e_{nn})=\lambda_{n-1,\id}(e_{11},\dots,e_{n-1,n-1})$$
 for all $1\leq i\leq n-1$, where $\sigma_i$ stands for the cycle $(i\,\, i+1\, \dots \,n)$. 
Accordingly, since $S_{n,n-1} $ consists of $\id$ and all  $\sigma_i$, $1\leq i\leq n-1$,
 the case $k=n-1$ follows from Claim 1. We may now assume that $k < n-1$ and that Claim 2 holds for  $k+1$. Take $\sigma\in S_{n,k}$. If $\sigma\in S_{n,k+1}$ then 
$$\lambda_{k\sigma}(e_{11},\dots,e_{kk})=-\lambda_{k+1,\sigma}(e_{11},\dots,e_{k+1,k+1})$$
 by Claim 1.
If $\sigma\not\in S_{n,k+1}$ there exists $1\leq i\leq k$ such that $\sigma(k+1)<\sigma(i)$.  Substituting 
 $$e_{11},\dots,e_{kk},e_{k+1,k+1},e_{k+1,k+2},e_{k+2,k+3},\dots,e_{n-1,n}$$
for $x_{\sigma(1)},\dots,x_{\sigma(n)}$  in $P$ we infer that for a certain permutation $\tau$ (specifically, $\tau=\sigma\circ (k+1\;k\;\dots\;i+1\;i)$) we have
$$\lambda_{k\sigma}(e_{11},\dots,e_{kk})=-\lambda_{k+1,\tau}(e_{11},\dots,e_{i-1,i-1},e_{k+1,k+1},e_{i+1,i+1},\dots,e_{k-1,k-1}).$$
 Since every $\lambda_{k\sigma}(e_{11},\dots,e_{kk})$ is associated to an evaluation of $\lambda_{k+1,\tau}$, Claim 2 follows by  the induction hypothesis and Claim 1.

\smallskip
{\bf Claim 3}.   $P=\lambda_{0,id} Q_n$.
\smallskip

By Claim 2 we have $\lambda_{0\sigma}=\lambda_{0\tau}$ for all $\sigma,\tau\in S_n$. Accordingly, $R := P-\lambda_{0,id} Q_n$ 
does not involve summands of the form $\mu x_{\sigma(1)}\ldots x_{\sigma(n)}$, $\mu\in F$, and can be therefore 
written as 
$$R(x_1,\dots,x_n)=\sum_{k=1}^{n}\sum_{\sigma\in S_{n,k}} \mu_{k\sigma}(x_{\sigma(1)},\dots,x_{\sigma(k)})x_{\sigma(k+1)}\cdots x_{\sigma(n)}.$$
We must prove that $R=0$, i.e., each $\mu_{k\sigma} =0$. We proceed by induction on $k$. For $k=0$ this holds by the hypothesis, so let $k > 0$. 
It suffices to show that $\mu_{k\sigma}(e_{i_1j_1},\dots,e_{i_kj_k})=0$ for arbitrary matrix units $e_{i_1j_1},\dots,e_{i_kj_k}$.
Choose distinct $l_1,\dots,l_{n-k}$ such that $l_s\ne i_t$ for all $s,t$.  
Substitute  $$e_{i_1j_1},\dots,e_{i_kj_k},e_{l_1l_2},e_{l_2l_3},\dots,e_{l_{n-k}i_1}$$ 
 for $x_{\sigma(1)},\dots,x_{\sigma(n)}$  in $P$.
 There is only one way to factorize $e_{l_1i_1}$ as a product of at most $n-k$  chosen matrix units, i.e.,  $e_{l_1i_1}=e_{l_1l_2}e_{l_2l_3}\cdots e_{l_{n-k}i_1}$. By induction hypothesis it thus follows that $\mu_{k\sigma}(e_{i_1j_1},\dots,e_{i_kj_k})=0$.
 % and since $e_{l_il_{i+1}}$ has to be followed by $e_{l_{i+1}l_{i+2}}$ by choice of $l_1,\dots,l_{n-k}$, no other product of $n-k$ matrix units from above equals $e_{l_1i_1}$. By induction hypothesis, we thus have $\mu_{k\sigma}(e_{i_1j_1},\dots,e_{i_kj_k})=0$.
%and $e_{l_1,i_1}=e_{l_1,l_2},e_{l_2,l_3}\cdots e_{l_{n-k},i_1}e_{i_1,j_1}$ if $i_1=j_1$. 
\end{proof}

\subsection{Specht problem for quasi-identities}
In this paragraph we finally prove a version of Specht problem for quasi-identities, i.e., we show that 
  $\mathfrak{I}_n$  is finitely generated as a T-ideal. 
 As it is well-known, Kemer \cite{Ke} has shown that for polynomial identities such a question has a positive answer (in characteristic $0$). 
  In our case the answer  is also positive since  we can apply the classical method of  {\em primary covariants} of Capelli and Deruyts (dressed up as Cauchy formula and highest weights), cf.  \cite[Chapter 3]{Proc3} to which we also refer for the statements used in the proof.
 \begin{theorem}
 The ideal $\mathfrak{I}_n$ of all quasi-identities of $M_n$  is finitely generated, as a T-ideal, 
 by elements which depend on at most  $2n^2$  variables.
\end{theorem}
 \begin{proof}
First of all,  if we impose the Cayley-Hamilton identity, we are reduced to study the problem  for the space $\mathcal C\langle X\rangle/(Q_n)$ isomorphic to $\ker\,\pi\subset   C_x\tnz_{\T_n}  \mathcal T_n\langle \xi_k\rangle \subset M_n(C_x\tnz_{\T_n}  C_y)$. Instead of considering all possible we consider only linear substitutions of variables,  that is,  we consider all these spaces as representations of the infinite linear group.

Now  we use  the language of symmetric algebras;  the space  $C_x$ equals  $S[M_n^*\otimes V]$ where $V=\oplus_{i=1}^\infty Fe_i$ is an infinite dimensional vector space  over which the infinite linear group $G_\infty$ acts. 

By Cauchy's formula  we have
$$S[M_n^*\otimes V]=\oplus_\lambda  S_\lambda (M_n^*)\otimes S_\lambda (V),$$ where $\lambda$ runs over all partitions with at most $n^2$  columns and $S_\lambda (V)$ is the corresponding Schur functor. 
By representation theory  the tensor product  $S_\lambda (V)\otimes S_\mu (V)$   of two such representations is a  sum of representations $S_\gamma  (V)$ where $\gamma  $ runs over  partitions with at most $2n^2$  columns.

Hence we have that $C_x\tnz_{\T_n}  C_y $, which is a quotient of  $C_x\tnz_{F}  C_y $,  is a sum of  representations $S_\lambda (M_n^*)\otimes S_\mu (V)$ where  $\mu$ has at most $2n^2$  columns.  Now any representation  $S_\gamma  (V)$ is irreducible under $G_\infty$  and generated by a highest weight vector.  If $\gamma$ has $k$ columns such a highest weight vector on the other hand lies in $S_\gamma  (\oplus_{i=1}^k Fe_i)$.

This means that under linear substitution of variables any element in $M_n(C_x\tnz_{\T_n}  C_y)$, and hence also in $ \ker\,\pi$,  is obtained from elements depending  on at most $2n^2$ variables.

Finally, if we restrict the number of variables to a finite number $m$, the corresponding space $ \ker\,\pi_m$ is a finitely generated module over a finitely generated algebra, and the claim follows.
\end{proof}

\bigskip

{\bf Acknowledgement.} The authors would like to thank  Hanspeter Kraft for an enlightening conversation and for providing some computer assisted computations.

\end{document}